\documentclass[11]{amsart}
\usepackage{amsfonts}
\usepackage{amsmath}
\usepackage{amsthm}
\usepackage{amssymb}
\usepackage{graphicx}
\usepackage{hhline}
\usepackage{makecell}
\usepackage{subfigure}
\usepackage{booktabs}
\usepackage{lscape}
\usepackage{grffile}
\usepackage{array}
\usepackage{caption}
\usepackage{longtable}
\usepackage{diagbox}
\usepackage{bm}
\usepackage{comment}
\usepackage{tikz}
\usepackage[normalem]{ulem}
\usepackage[bindingoffset=0.2in,%
left=1.0in,right=1.0in,top=1.25in,bottom=1.375in,%
            footskip=.25in]{geometry}

\usepackage[hyperfootnotes=false, colorlinks, linkcolor={blue}, citecolor={magenta}, filecolor={blue}, urlcolor={blue}, plainpages=false, pdfpagelabels]{hyperref}

\usepackage[toc]{appendix}

\newtheorem{theorem}{Theorem}
\newtheorem{proposition}[theorem]{Proposition}

\newtheorem{definition}[theorem]{Definition}
\newtheorem{lemma}[theorem]{Lemma}
\theoremstyle{remark}
\newtheorem{remark}{Remark}
\newcommand{\p}{\mathfrak{p}}

\newcolumntype{C}[1]{>{\centering\arraybackslash}p{#1}} 
\usepackage{hyperref}
\usepackage{soul} 

\newcommand{\fp}{\mathfrak p}
\newcommand{\nup}{\nu_{\mathfrak{p}}}

\newcommand{\OK}{\mathcal{O}_K}

\newcommand{\minD}{\mathfrak{D}^{\text{min}}}

\newcommand{\Ci}{\mathcal{C}_{p^2, i}(t, d)}

\title[Towards a classification of $p^2$-Discriminant Ideal Twins]{Towards a classification of $p^2$-Discriminant Ideal Twins over Number Fields}
\date{}

\begin{document}

\author[A. Deines]{Alyson Deines}
\address{Center for Communications Research, San Diego, CA, USA}
\email{aly.deines@gmail.com}

\author[A. S. Hamakiotes]{Asimina S. Hamakiotes}
\address{University of Connecticut, Storrs, CT, USA}
\email{asimina.hamakiotes@uconn.edu}

\author[A. Iorga]{Andreea Iorga}
\address{Department of Mathematics, Cornell University, Ithaca, NY, USA}
\email{ai324@cornell.edu}

\author[C. Namoijam]{Changningphaabi Namoijam}
\address{Department of Mathematics, Colby College, Waterville, ME, USA}
\email{cnamoijam@gmail.com}

\author[M. Roy]{Manami Roy}
\address{Department of Mathematics, Lafayette College, Easton, PA, USA}
\email{royma@lafayette.edu}

\author[L. D. Watson]{Lori D. Watson}
\address{Department of Mathematics, Trinity College, Hartford, CT, USA}
\email{lori.watson@trincoll.edu}


\begin{abstract}
Isogenous elliptic curves have the same conductor but not necessarily the same minimal discriminant ideal. 
In this article, we explicitly classify all $p^2$-isogenous elliptic curves defined over a number field with the same minimal discriminant ideal for odd prime $p$ where $X_0(p^2)$ has genus 0, i.e., $p = 3$ or $5$. As a consequence, we give a list of all $p^2$-isogenous discriminant (ideal) twins over $\mathbb{Q}$ for such $p$.

\end{abstract}
\maketitle

\setcounter{tocdepth}{1}
\tableofcontents

\section{Introduction}

Isogenous elliptic curves defined over a number field have the same conductor; however, their minimal discriminant ideals might be different. 
We are interested in understanding when two (cyclic) isogenous elliptic curves also have the same minimal discriminant ideal, i.e., when they are discriminant ideal twins. 
This work continues the classification of discriminant (ideal) twins begun by the first author in  \cite{Deines2018} and continued in \cite{RNT1}.
Discriminant ideal twins serve as an obstruction to determining the optimal quotient of the modular parameterization of an elliptic curve by a modular or Shimura curve using the algorithm given by Ribet and Takahashi in \cite{Ribet1997ParametrizationsOE}. 
In~\cite{Deines2018}, the first author determines all such obstructions for elliptic curves $E$ defined over $\mathbb{Q}$ with at least one prime of multiplicative reduction, providing a list of all $j$-invariants corresponding to semistable isogenous\footnote{Throughout, we will restrict to cases of cyclic isogenies.} discriminant twins. 
This result is extended to (non-semistable) $p$-isogenous elliptic curves defined over arbitrary number fields for $p \in \{3, 5, 7, 13\}$ by Barrios,  Brucal-Hallara, Deines, Harris, and Roy in \cite{RNT1}. 
With a goal of completing the classification for other isogeny degrees we further extend this work by determining $p^2$-isogenous discriminant ideal twins over arbitrary number fields in the case where $p$ is an odd prime and $X_0(p^2)$ has genus $0$, i.e., $p^2 \in \{9, 25\}$.

If $E_1$ and $E_2$ are (cyclic) isogenous discriminant ideal twins, then we can create infinitely many more by twisting at primes $\ell$ coprime to $2$ and the conductor of the curves.
Note that twisting preserves $j$-invariants and so, even though twisting by such $\ell$ creates infinitely many pairs of isogenous elliptic curves with the same discriminant ideals, it does not create new twist families.

It is easy to check the Cremona database (via SageMath \cite{sagemath} or LMFDB \cite{lmfdb}) and find all twist families of isogenous discriminant ideal twins over $\mathbb{Q}$.  For each prime isogeny degree, all the $j$-invariant pairs found in the Cremona database are given in  \cite[Table 5.1]{RNT1}.
For each $p^2$-isogeny degree, all the $j$-invariant pairs found in the Cremona database are given in Table~\ref{QQdisctwins}.

\begin{table}
\caption{$j$-invariant pairs for $p^2$-isogenous elliptic curves}
\label{QQdisctwins}
$$\begin{array}{rcc}\hline
    \text{deg} & j_1 & j_2 \\ \hline
 4 & -1/15 & 56667352321/15\\
   &  35937/17 & 82483294977/17\\
   &  287496 & 287496\\ \hline
    9 & -50357871050752/19 & 32768/19\\
  &  4096000/37 & 727057727488000/37\\ \hline
    25 & -52893159101157376/11 & -4096/11\\
  &  190705121216/71 & 3922540634246430781376/71\\ \hline
\end{array}$$
\end{table}

The next question to ask is, are there more $j$-invariant pairs of $p$ or $p^2$-isogenous discriminant ideal twins? 
The Cremona tables in LMFDB contain curves with conductor up to 299,996,953 and are complete for curves with conductor up to 500,000.  
With more data, would we find more discriminant ideal twins?
For isogeny degrees $p = 2, 3, 5, 7, 13$ over $\mathbb{Q}$, the answer is no.
This was first examined by the first author~\cite{Deines2018} for elliptic curves with at least one prime with multiplicative reduction, and the more general prime isogeny case for  elliptic curves defined over number fields was completed by Barrios, Brucal-Hallara, Deines, Harris, and Roy~\cite{RNT1}.
Similarly, the first author \cite{Deines2018} also examined the cases $p^2 = 4, 9, 25$ for curves with at least one prime of multiplicative reduction. 

In this paper, we continue this exploration and finish certain cases over $\mathbb{Q}$ for curves without multiplicative reduction; we also work with curves with both multiplicative and non-multiplicative reduction over number fields, in general. 

As in previous work, all isogenies are cyclic.
Specifically, we examine when isogenous discriminant ideal twins defined over a number field $K$ occur for $p^2$-isogeny degrees coming from genus 0 modular curves for odd $p$, i.e.,  for isogeny degrees $n = 9, 25$. In \cite{Bariso}, Barrios gives a parameterization of $p^2$-isogenous elliptic curves with the property that their $j$-invariants are not both $0$ and $1728$. In particular, for $p = 3, 5$, if $E_1$ and $E_3$ are two $p^2$-isogenous elliptic curves then there exist $t\in K$, $d \in K^\times/K^{\times2}$ such that $E_i \cong \mathcal{C}_{p^2, i}(t, d)$ (see Table~\ref{ta:curves} for precise definition of $\mathcal{C}_{p^2, i}(t, d)$). Using this notation, we state our main results: 
\begin{theorem}\label{n=9_iff}
Let $E_1$ and $E_3$ be two $9$-isogenous elliptic curves over a number field $K$ such that their $j$-invariants are not equal.
Suppose further that $E_i \cong \mathcal{C}_{9, i}(t, d)$, where $t\in \mathcal{O}_K$ and $d\in \mathcal{O}_K$ given in Table~\ref{ta:curves}. Then $E_1$ and $E_3$ are discriminant ideal twins if and only if for each prime $\fp$ of $\mathcal{O}_K$,
    \begin{equation}\label{eq:n=9}
        \nup(t-3) = 3k_\fp, \quad \text{ for } 0 \leq k_\fp \leq \nup(3). 
    \end{equation}
    Moreover, the two curves are discriminant twins if and only if $t$ satisfies \eqref{eq:n=9} and 
    $(t-3)^8\in\mathcal{O}_{K}^{12}$.
\end{theorem}

\begin{theorem}\label{n=25_iff}
    Let $E_1$ and $E_3$ be two $25$-isogenous elliptic curves over a number field $K$ such that their $j$-invariants are not equal. Suppose further that  $E_i \cong \mathcal{C}_{25, i}(t, d)$, where $t\in \mathcal{O}_K$ and $d\in \mathcal{O}_K$ given in Table~\ref{ta:curves}. Then $E_1$ and $E_3$ are discriminant ideal twins if and only if for each prime $\p$ of $\mathcal{O}_K$,
    \begin{equation}\label{eq0:n=25}
        \nup(t-1) = k_\p, \quad \text{ for } 0 \leq k_\p \leq \nup(5).
    \end{equation}
   In fact, the two curves are discriminant twins  if $t$ satisfies \eqref{eq0:n=25}.
\end{theorem}

We prove Theorem~\ref{n=9_iff} (see Theorems~\ref{n=9_if}~and~\ref{n=9_onlyif}) in Section~\ref{proof_for_n=9} and Theorem~\ref{n=25_iff} (see Theorems~\ref{n=25_if}~and~\ref{n=25_onlyif}) in Section~\ref{proof_for_n=25}.
The proofs of these theorems rely on the fact that we can take $t \in \mathcal{O}_K$, not just in $K$, which is proved in Lemma~\ref{Lem:tisOKint} of Section~\ref{results_for_p^2}. 
The rest of this paper is organized as follows. 
In Section~\ref{preliminaries}, we recall preliminary definitions and results about elliptic curves.  
Section~\ref{results_for_p^2} discusses the cases where the curves have the same $j$-invariant. We treat the case when the two curves have $j$-invariants both equal to $0$ or $1728$ separately.
As a consequence of our main theorems, in Section~\ref{sec:Examples} it is confirmed that, over $\mathbb{Q}$, the Cremona database contains all the $j$-invariant pairs associated to $p^2$-isogenies, for odd $p$. 
Examples and code are available on github \cite{GitHubDIT}.

\subsection*{Acknowledgments}  
The authors thank the Banff International Research Station (BIRS) and Women in Numbers 6 (23w5175) for the opportunity to initiate this collaboration. The authors would like to thank the referee for valuable suggestions and comments, and for identifying an error in an earlier version of the paper. CN was partially supported by a Colby College Research Grant. The AMS Simons Travel Grant program partially supported MR during this work.

\section{Preliminaries}\label{preliminaries}

In this section we will recall some basic definitions and facts about elliptic curves; see Silverman~\cite{Silverman2009} and  \cite[section 2]{RNT1} for details. To start, let $K$ be a number field or a local field of characteristic $0$. An elliptic curve $E$ defined over $K$ is given by a Weierstrass model
\begin{equation}\label{eq:1}
E:y^2 + a_1 xy + a_3 y = x^3 + a_2 x^2 + a_4x + a_6,     
\end{equation}
with $a_i \in K.$  The $b_i$ and $c_i$-invariants are further defined as
$$
\begin{array}{ll}
 b_2 = a_1^2 + 4a_2, & b_4 = 2a_4 + a_1 a_3, \qquad \qquad \qquad b_6 = a_3^2 + 4a_6,  \\
 c_4 = b_2^2 - 24b_4, & c_6 = -b_2^3 + 36b_2 b_4 - 216b_6,\\
 \Delta = (c_4^3 - c_6^2)/1728, & j = c_4^3 /\Delta. \\
\end{array}
$$ 
We say $E$ is given by an 
\textit{integral Weierstrass model} if each $a_i$ is in the ring of integers of $K$ 
for $i \in \{1, 2, 3, 4, 6\}$.    
 
An elliptic curve $E^\prime$
is $K$-isomorphic to $E$ if there is an \emph{admissible change of variables} $\tau: E \rightarrow E'$ defined by $(x,y) \mapsto (u^2 x + r,u^3 y + su^2x + w)$, where $u, r, s, w \in K$, with $u \neq 0$. We write $\tau = [u, r, s, w].$ Consequently,
\[
j^{\prime} = j, \qquad \Delta^{\prime}=u^{-12}\Delta,\qquad c_{4}^{\prime}%
=u^{-4}c_{4},\qquad c_{6}^{\prime}=u^{-6}c_{6}.
\] 
If any of $u, r, s, w \not \in K$, then $E$ and $E^\prime$ are isomorphic over an algebraic closure of $K$. When this occurs, i.e., when $\tau$ is defined over an extension of $K$, we say that $E^\prime$ is a \textit{twist} of $E$.

In the case that $K$ is a number field with ring of integers $\OK$, 
let $\fp$ be a prime ideal of $K$ and $\nup$ be the normalized valuation of the completion $K_{\fp}$ of $K$ at~$\fp$. 
Further, let $R_{\fp}$ be the ring of integers of $K_{\fp}$.
By choosing an element of $\fp \setminus \fp^2$, we can find an element $\pi \in \mathcal{O}_K$ that has $\nup(\pi) = 1$ and $\pi$ is thus a uniformizer in $R_{\fp}$.

Consider $E$ defined over $K$ with Weierstrass equation as in \eqref{eq:1}.
We view $E$ as being defined over $K_{\fp}$ via the inclusion $\iota : K \hookrightarrow K_{\fp}$.
Via this inclusion we often conflate $E$ being defined over $K$ with $E$ being defined over $K_{\fp}$.
If $E$ is not already $\fp$-integral, then there is at least one coefficient $a_i$ with negative valuation.
The transformation $\tau = [u^{-1}, 0, 0, 0]$ on $E$ gives a $K_{\fp}$-isomorphic elliptic curve whose Weierstrass coefficients are $u^{i}a_i$. 
Thus, choosing $u^{-1}$ to be an appropriate power of $\pi$ we can find a model for $E$ such that all coefficients $u^ia_i$ are in $R_{\fp}$, i.e., we can always find a $\fp$-integral Weierstrass model.
As we chose $\pi$ to be an element of $\mathcal{O}_K$, the model $\tau E$ is defined over $K$ in addition to $K_{\fp}.$

A $\fp$-integral Weierstrass model as constructed above will, by construction, have an integral discriminant $\nup(\Delta) \geq 0$.
As $\nup$ is discrete, there will be a $\fp$-integral model such that $\nup(\Delta) \geq 0$ is minimal. 
This model will not be unique, but the minimal valuation of the discriminant will be.

\begin{definition}
    A Weierstrass model for $E$ defined over $K_{\fp}$ is called a \textit{$\fp$-minimal (Weierstrass) model} if $\nup(\Delta)$ is minimized subject to the constraint that the model is a $\fp$-integral Weierstrass model. Any $\Delta$ such that $\nup(\Delta)$ is minimal is called a \textit{$\fp$-minimal discriminant} of $E$, and we call $\nup(\Delta)$ the \textit{valuation of the minimal discriminant} of $E$ at $\fp$. 
\end{definition}

\begin{remark}
    While we have been working prime by prime, we note that we can always find a $\fp$-integral model that is globally integral.
    By \cite[Proposition 4.7.8]{CohenCANT}, we can always construct $\beta \in \mathcal{O}_K$ that has prescribed (non-negative) valuations at a finite set of prime ideals and non-negative valuation at all other primes.
    Using this element, we can transform $E$ to an integral model $\tau E$ that is integral and $\fp$-minimal at the prescribed set of primes.
    The caveat is that the discriminant of $\tau E$ is not necessarily minimal at primes not in the prescribed set.
\end{remark}

When working with elliptic curves over local fields, we will also consider the $c_i$-invariants along with the discriminant. The following definition will be used frequently in the later sections.
\begin{definition}
For $\fp$ a prime of $K$, we define the signature of $E$ defined over $K$ with
respect to $\fp$ to be%
\[
\text{sig}_{\fp} \left(  E\right)  =\left(  \nup\left(
c_{4}\right)  , \nup\left(  c_{6}\right)  , \nup \left(  \Delta\right)
\right).
\]
\end{definition}

Next, we look at the global properties of elliptic curves related to the valuation of the minimal discriminant of $E$ at $\fp$. Using the embedding  $\iota$, for each prime $\fp$ of $K$, we can view an elliptic curve $E$ defined over $K$ as an elliptic curve
defined over $K_{\fp}$.
Then, together with \cite[Lemma 2.6]{RNT1} and the above discussion, we have the following global definition.
\begin{definition}
Let $E$ be an elliptic curve defined over $K$. For each prime $\fp$ of $\mathcal{O}_K$ we can choose a $\fp$-minimal model given by an integral Weierstrass model that is $K$-isomorphic to $E$. This $\fp$-minimal model $E_{\fp}$ has discriminant $\Delta_\fp$, which is defined to be the $\fp$-minimal discriminant of $E$. 
    The \textit{minimal discriminant ideal} of $E$ defined over $K$, denoted by $\minD$, is the (integral) ideal of $K$ given by 
    $$\minD = \prod_{\fp} \fp^{\nup(\Delta_{\fp})}.$$
\end{definition}

Over a number field, it is not always possible to find a single Weierstrass equation that is simultaneously minimal for every prime $\fp$, as illustrated in \cite[Example 2.8]{RNT1}. When it is possible, we have a global minimal model as follows. 

\begin{definition}
    A \textit{global minimal model} for $E$ defined over $K$ is an integral Weierstrass equation 
    $$y^2 + a_1 xy + a_3 y = x^3 + a_2 x^2 + a_4 x + a_6$$
    such that the discriminant $\Delta$ of the equation satisfies $\minD = \left( \Delta \right).$
\end{definition}

The conductor of an elliptic curve is closely related to its discriminant, and measures, in some sense, the arithmetic complexity of the curve. The conductor describes the possible bad reduction types of an elliptic curve. We will skip the definition here and refer to Silverman~\cite[Section VIII.11]{Silverman2009}. Next, we define the object of study, discriminant (ideal) twins.

\begin{definition}
\label{def:disctwins}
Let $K$ be a number field, and let $E$ and $E'$ be elliptic curves defined over $K$ that are not $K$-isomorphic. We say that $E$ and $E'$ are \textit{discriminant ideal twins} if they have the same minimal discriminant ideal and the same conductor.    If, additionally, for each prime $\fp$ there exist $\fp$-minimal models for $E$ and $E'$ defined over $\OK$ such that $\Delta_{\fp} = \Delta'_{\fp}$,    then we say $E$ and $E'$ are \textit{discriminant twins.}
\end{definition}

Now that we have defined discriminant (ideal) twins, we can investigate the discriminant ideals of $n$-isogenous elliptic curves over number fields. 
Two elliptic curves $E_1$ and $E_2$ defined over the same field $K$ are \textit{isogenous} if there exists a non-constant morphism $\varphi \colon E_1 \rightarrow E_2$ with coefficients in $K$ mapping the neutral element of $E_1$ to the neutral element of $E_2$. An equivalent definition states that an isogeny is a morphism $\varphi \colon E_1 \to E_2$ that is surjective and has finite kernel, $\ker \varphi$. If $\ker\varphi\cong \mathbb{Z} /n \mathbb{Z}$ for a natural number $n$, we say $E_1$ and $E_2$ are $n$-isogenous. Furthermore, the isogeny $\varphi$ is defined over $K$ when $\ker\varphi$ is $\operatorname*{Gal}\!\left(\overline{K}/K\right)  $-invariant. 
For an $n$-isogeny   $\varphi:E_1\rightarrow E_2$ 
defined over $K$, the $K$-isomorphism class of $(E,\ker \varphi)$ is a non-cuspidal $K$-rational point on the classical modular curve $X_0(n)$. Note that the non-cuspidal $K$-rational points of $X_0(n)$ parameterize isomorphism classes of pairs $(E,C)$ where $E$ is an elliptic curve defined over $K$ and $C$ is a cyclic subgroup of $E$ of order $n$ such that $C$ is $\operatorname*{Gal}\!\left(\overline{K}/K\right)  $-invariant. 

In this article we focus on some $n$-isogenies when $X_0(n)$ has genus $0$. In particular we focus on the case when $n=p^2$, for an odd prime $p$. If $E_1 \to E_3$ is a $p^2$-isogeny, then there is an elliptic curve $E_2$ such that  $E_i \to E_{i+1}$ is a $p$-isogeny for $i = 1, 2$. Here, $E_1$, $E_2$ and $E_3$ are non-isomorphic curves. Theorem~\ref{parafamilies} gives a parameterization of such curves $E_i$. Before introducing Theorem \ref{parafamilies}, let us define some polynomials $T_i, S_i$ in Table \ref{ta:polyTi} that will be used frequently in this article. 
{
\renewcommand*{\arraystretch}{1.2} 
	\renewcommand{\arraycolsep}{0.4cm}
	\begin{longtable}{ccC{3.7in}}
	\caption{Polynomials in $t$ that appear in models for $\mathcal{C}_{p^2,i}(t,d)$}\\
		$p^2$&\text{Notation} & \text{Polynomial} \\ \hline
		\endfirsthead
	\caption{Polynomials in $t$ that appear in models for $\mathcal{C}_{p^2,i}(t,d)$}\\
	$p^2$&\text{Notation} & \text{Polynomial} \\ \hline
		\endhead
		\hline
		\multicolumn{3}{r}{\emph{continued on next page}}
		\endfoot
		\hline
		\endlastfoot
		$9$&$T_{1}$ & $t-3$ \\
		\cmidrule{2-3}
		&$T_{2}$ & $t+6$ \\
		\cmidrule{2-3}
		&$T_{3}$ & $t^2+3t+9$ \\
		\cmidrule{2-3}
		&$T_{4}$& $t^3-24$ \\
		\cmidrule{2-3}
		&$T_{5}$& $t^3+234t^2+756t+2160$ \\
			\cmidrule{2-3}
		&$T_{6}$ & $t^6-36t^3+216$ \\ 	\cmidrule{2-3}
		&$T_{7}$ & $t^6-504t^5-16632t^4-123012t^3-517104t^2-1143072t-1475496$ \\	\cmidrule{2-3}
		&$T_{8}$ & $t$\\ \cmidrule{2-3}
        &$T_9$ & $t^2 - 6t+36$ \\ \cmidrule{2-3}
        &$T_{10}$ & $t^2-6t-18$ \\ \cmidrule{2-3}
        &$T_{11}$ & $t^4+6t^3+54t^2-108t+324$ \\
  \hline

		$25$&$S_{1}$ & $t^4 + t^3 + 6t^2 + 6t + 11$ \\ \cmidrule{2-3}
		&$S_{2}$ & $t^2 + 4$ \\ \cmidrule{2-3}
		&$S_{3}$ & $t - 1$ \\ 		\cmidrule{2-3}
		&$S_{4}$ & $\makecell{t^{10} + 240t^9 + 2170t^8 + 8880t^7 + 34835t^6 +  83748t^5 \\+ 206210t^4 + 313380t^3 + 503545t^2 + 424740t + 375376}$ \\ \cmidrule{2-3}
		&$S_{5}$ & $t^4 + 6t^3 + 21t^2 + 36t + 61$ \\ 		\cmidrule{2-3}
		&$S_{6}$ & $\makecell{t^{10} - 510t^9 - 13580t^8 - 36870t^7 - 190915t^6 - 393252t^5 \\- 1068040t^4 - 1508370t^3 - 2581955t^2 - 2087010t - 1885124}$ \\ 		\cmidrule{2-3}
		&$S_{7}$ & $t^{10} + 10t^8 + 35t^6 - 12t^5 + 50t^4 - 60t^3 + 25t^2 - 60t + 16$ \\ \cmidrule{2-3}
		&$S_{8}$ & $t^4 + 3t^2 + 1$ \\ \cmidrule{2-3}
		&$S_{9}$ & $t^{10} + 10t^8 + 35t^6 - 18t^5 + 50t^4 - 90t^3 + 25t^2 - 90t + 76$ \\ \cmidrule{2-3}
        &$S_{10}$ & $t^2 + 3t + 1$ \\ \cmidrule{2-3}
        &$S_{11}$ & $t^4-4t^3+11t^2-14t+31$ \\ \cmidrule{2-3}
        &$S_{12}$ & $t^4 + t^3 + 11t^2 - 4t + 16$ \\ \cmidrule{2-3}
        &$S_{13}$ & $t^2-2t-4$ \\ \cmidrule{2-3}
        &$S_{14}$ & $t^4-4t^3+21t^2-34t+41$ \\
\label{ta:polyTi}
	\end{longtable}
}
\begin{theorem}[{Barrios \cite[Theorem~1]{Bariso}}] \label{parafamilies} Let $p$ be an odd prime such that $X_0(p^2)$ has genus $0$ and let $K$ be a number field or a local field of characteristic $0$. Let $E_1$ and $E_3$ be elliptic curves defined over $K$ such that the $j$-invariants of $E_1$ and $E_3$ are not both identically $0$ or $1728$. Suppose further that $E_1$ and $E_3$ are $p^2$-isogenous elliptic curves over $K$ 
so that there is an elliptic curve $E_2$ that is $p$-isogenous to both $E_1$ and $E_3$. Then there are $t \in K$ and $d \in \OK$ such that $E_i$  are $K$-isomorphic to $\Ci$ given in Table~\ref{ta:curves} for $i \in \{1,2,3\}$, respectively.
\end{theorem}
{\renewcommand*{\arraystretch}{1.2} 
	 \begin{longtable}{ccC{2.6in}C{2.7in}}
		\caption{The elliptic curve  $\Ci:y^2=x^3+d^2 A_{p^2,i}(t)x+d^3 B_{p^2,i}(t)$}\\
		\hline
		$p^2$ & $i$ & $A_{p^2,i}(t)$ & $B_{p^2,i}(t)$\\
		\hline
		\endfirsthead
		\caption{The elliptic curve $\Ci:y^2=x^3+d^2 A_{p^2,i}(t)x+d^3 B_{p^2,i}(t)$}\\
	\hline
	$p^2$ & $i$ & $A_{p^2,i}(t)$ & $B_{p^2,i}(t)$\\
		\hline
		\endhead
		\hline
		\multicolumn{3}{r}{\emph{continued on next page}}
		\endfoot
		\hline
		\endlastfoot	
		$9$ & $1$ & $-3 \cdot T_{2}\cdot T_{5}$ & $-2 \cdot T_{7}$\\\cmidrule{2-4}
		& $2$ & $ -3^5  \cdot T_2  \cdot T_{8}\cdot  T_9$  & $-2 \cdot 3^6\cdot  T_{10} \cdot T_{11} $ \\\cmidrule{2-4}
		& $3$ & $-3^9 \cdot T_4 \cdot T_{8} $ & $ -2 \cdot 3^{12}\cdot T_{6}$  \\\hline
$25$ & $1$ & $-3^3  \cdot S_{2}\cdot S_{4}$ & $-2 \cdot 3^3   \cdot S_{2}^{2}\cdot S_5  \cdot S_{6} $ \\\cmidrule{2-4}
& $2$ & $-3^3  \cdot 5^4 \cdot  S_{2}  \cdot S_{10} \cdot S_{11} \cdot S_{12} $ & $-2  \cdot 3^3  \cdot 5^6 \cdot S_{2}^{2} \cdot S_5  \cdot S_{8}\cdot S_{13} \cdot S_{14}$ \\\cmidrule{2-4}
& $3$ &  $-3^3  \cdot 5^8\cdot  S_{2} \cdot  S_{7} $ & $ -2  \cdot 3^3  \cdot 5^{12} \cdot   S_{2}^{2}   \cdot   S_{8}  \cdot   S_{9}$
		\label{ta:curves}
\end{longtable}
}

\begin{remark}\label{rmk:dinOK}
We note that Theorem~\ref{parafamilies} differs slightly from \cite[Theorem 1]{Bariso} in that the latter is stated more generally for any field $K$ of characteristic relatively prime to $6p^2$, and $d \in K^\times/(K^\times)^2$. In this article, we will only consider number fields or local fields of characteristic $0$. 
In this setting, if $d \in K^\times/(K^\times)^2$ and $d \not \in \mathcal{O}_K$, then we may take a different representative $d' \in \mathcal{O}_K$ such that $\mathcal{C}_{p,i}(t,d)$ is $K$-isomorphic to $\mathcal{C}_{p,i}(t,d')$.
The curves in \cite[Theorem 1]{Bariso} $C_{n, i}(t, d)$ are given in the form $y^2 = x^3 + d^2 A_{n, i}(t) x + d^3 B_{n, i}(t)$.
Let $\{\fp_j\}$ be the finite set of primes such that $\nu_{\fp_j}(d) = - e_j$ where $e_j > 0$.
By \cite[Proposition 4.7.8]{CohenCANT}, we can take $\beta \in \mathcal{O}_K$ such that $\nu_{\fp_j}(\beta) = e_j$.
The isomorphism $\tau = [1/\beta, 0, 0, 0]$ yields $\tau C_{n, i}(t, d): y^2 = x^3 + d^2 \beta^4 A_{n, i}(t) x + d^3 \beta^6 B_{n, i}(t)$.
Thus setting $d' = d \beta^2$ gives $d' \in \mathcal{O}_K$ and $d' = d$ in $K^{\times}/(K^{\times})^2$.
\end{remark}

Next, Table~\ref{ta:invariants} gives the $j$-invariants, $c_i$-invariants, and discriminants of $\mathcal{C}_{p^2,i}(t,d)$ for $i \in \{1, 2, 3\}$, which will be helpful in the work that follows.  As $d$ is a twisting parameter, we see that $j_{p, i}(t) = j_{p, i}(t, d)$ does not depend on $d$. 

{\renewcommand*{\arraystretch}{1.8} 
	\renewcommand{\arraycolsep}{0.4cm}
	\begin{longtable}{ccC{1in}C{1.1in}C{1.3in}C{1.3in}}
		\caption{The quantities $d^{-2}c_{4,i}(t,d)$, $d^{-3}c_{6,i}(t,d)$, $d^{-6} \Delta_{i}(t,d)$, and $j_{i}(t)$} \\
		\hline
		$p^2$ & $i$  &$d^{-2}c_{4,i}(t,d)$ &$d^{-3}c_{6,i}(t,d)$& $d^{-6}\Delta_{i}(t,d)$& $j_{i}(t)$ \\
		\hline
		\endfirsthead
		\caption{The quantities $j_{i}(t)$ and $2^{-12} 3^{-12}d^{-6} \Delta_{i}(t,d)$} \\
		\hline
		$p^2$ & $i$  &$d^{-2}c_{4,i}(t,d)$ &$d^{-3}c_{6,i}(t,d)$& $d^{-6}\Delta_{i}(t,d)$& $j_{i}(t)$ \\
		\hline
		\endhead
		\hline
		\multicolumn{5}{r}{\emph{continued on next page}}
		\endfoot
		\hline
		\endlastfoot	
		$9$ & $1$ & $2^4 \cdot 3^2 \cdot T_{2} \cdot T_{5}$ & $2^6 \cdot 3^3\cdot T_{7}$ & $2^{12}\cdot 3^6 \cdot T_{1}^9 \cdot T_{3}$& $T_{2}^3\cdot T_{5}^3\cdot T_{1}^{-9}\cdot T_{3}^{-1}$ \\\cmidrule{2-6}
		& $2$ & $2^4 \cdot 3^6 \cdot T_2 \cdot T_8 \cdot T_9$ & $2^6 \cdot 3^9 \cdot T_{10} \cdot T_{11}$  & $2^{12} \cdot 3^{18}  \cdot T_1^3 \cdot T_3^3$ & $T_2^3 \cdot T_8^3 \cdot T_9^3  \cdot T_1^{-3} \cdot T_3^{-3}$\\\cmidrule{2-6}
		& $3$ &  $2^4 \cdot 3^{10}  \cdot T_{4} \cdot T_{8}$  & $2^6 \cdot 3^{15}\cdot T_{6}$& $2^{12} \cdot 3^{30}\cdot T_{1} \cdot T_{3}$& $ T_{4}^3 \cdot T_8^3\cdot  T_1^{-1} \cdot T_{3}^{-1}$ \\\hline
		
		$25$ & $1$ & $2^4 \cdot 3^4 \cdot S_{2} \cdot S_{4}$ & $2^6 \cdot 3^6 \cdot S_{2}^2 \cdot S_{5} \cdot S_{6}$& $2^{12}\cdot 3^{12}\cdot S_{1} \cdot S_{2}^3 \cdot S_{3}^{25}$ & $ S_{4}^3  \cdot S_{1}^{-1} \cdot S_{3}^{-25}$  \\\cmidrule{2-6}
		& $2$ & $2^4 \cdot 3^4 \cdot 5^4 \cdot S_2 \cdot S_{10} \cdot S_{11} \cdot S_{12}$ & $2^6 \cdot 3^6 \cdot 5^6 \cdot S_2^2 \cdot S_5 \cdot S_8 \cdot S_{13} \cdot S_{14}$ & $2^{12} \cdot 3^{12} \cdot 5^{12} \cdot S_1^5 \cdot S_2^3 \cdot S_3^5$ & $S_{10}^3 \cdot S_{11}^3 \cdot S_{12}^3 \cdot S_1^{-5} \cdot S_3^{-5}$ \\\cmidrule{2-6}
		& $3$ &  $2^4 \cdot 3^4 \cdot 5^8\cdot S_{2} \cdot S_{7}$ & $2^6 \cdot 3^6 \cdot 5^{12} \cdot S_{2}^2 \cdot S_{8} \cdot S_{9}$&$2^{12}\cdot 3^{12}\cdot 5^{24}\cdot S_{1} \cdot S_{2}^3 \cdot S_{3}$ & $S_{7}^3\cdot S_{1}^{-1}  \cdot S_{3}^{-1} $
\label{ta:invariants}
	\end{longtable}
}
\section{\texorpdfstring{Results for $n=p^2$ for odd primes $p$}{}}\label{results_for_p^2}
In this section, we prove a few results for $p^2$-isogenous curves where $p \in \{3, 5\}$ that are used to establish our main results. Let $K$ be a number field. Suppose that $E_1$ and $E_3$ are $p^2$-isogenous elliptic curves over $K$ 
so that there is an elliptic curve $E_2$ that is $p$-isogenous to both $E_1$ and $E_3$. Note that the curves $E_1$, $E_2$ and $E_3$ are non $K$-isomorphic curves. The following result ensures that one can assume $t \in \OK$ when using parametrized families in Theorem~\ref{parafamilies} for discriminant ideal twins.

\begin{lemma}
\label{Lem:tisOKint}
Let $p$ be an odd prime such that $X_0(p^2)$ has genus $0$.
Suppose that $E_1$ and $E_3$ are $p^2$-isogenous discriminant ideal twins over $K$.
Then there exists $t \in \mathcal{O}_K$ and $d \in \mathcal{O}_K$ such that $E_i \cong \mathcal{C}_{p^2,i}(t, d)$.
\end{lemma}

\begin{proof}
Since $E_1$ and $E_3$ are $p^2$-isogenous, it follows from Theorem~\ref{parafamilies} that there exist elements $t \in K$ and $d \in \mathcal{O}_K$ such that $E_i \cong \mathcal{C}_{p^2, i}(t, d)$. It remains to show that $t \in \mathcal{O}_K$ when $E_1$ and $E_3$ are discriminant ideal twins.

Towards a contradiction, assume that there exists a prime $\fp$ of $\mathcal{O}_K$ such that $\nup(t) = -k<0$, for some $k \in \mathbb{Z}$.
Given such $t$, for a monic polynomial $f(t)=t^n+a_{n-1}t^{n-1}+\cdots+a_0 \in \mathbb{Z}[t]$, then $\nup(f(t)) =\min(t^n,a_{n-1}t^{n-1},\cdots, a_0)=-kn= -k \cdot \deg(f)$.
Looking at the $j$-invariant of $E_1$ when $p = 3$ in Table~\ref{ta:invariants}, we obtain: 
\begin{align*}
    \nup(j(E_1)) &= \nup(T_2^3\cdot T_5^3 \cdot T_1^{-9} \cdot T_3^{-1}) \\
    &= 3\nup(T_2) + 3\nup(T_5) - 9\nup(T_1) - \nup(T_3) \\
    &= -3k - 9k + 9k + 2k = -k,
\end{align*}
where the third equality follows as $T_i \in \mathbb{Z}[t]$ are monic. 
By similar arguments, using Table~\ref{ta:invariants}, for both $p = 3$ and $p = 5$ we obtain: 
\begin{align*}
        \nup(j(E_1)) &= \nup(t) = -k, \\
    \nup(j(E_2)) &= p\nup(t) = -pk, \\
    \nup(j(E_3)) &= p^2\nup(t) = -p^2k. 
    \end{align*}
    It follows that the three curves must have either multiplicative reduction or potentially multiplicative reduction at $\fp$. Let $u_i \in K^\times$ be such that 
    the curve $\tau_i E_i$ is minimal at $\fp$ for $\tau_i = [u_i, 0, 0, 0]$. Let $\Delta(E_{i,\fp})$ be the minimal discriminant of $E_i$ at $\fp$. 

    If the curves have multiplicative reduction, then from \cite[Table 1]{DokchitserDokchitser2015} we know that 
    \begin{align*}
        \nup(\Delta(E_{2, \fp})) = p\nup(\Delta(E_{1, \fp})) \text{ as } \nup(j(E_2)) = p \nup(j(E_1))
    \end{align*}
    and
    \begin{align*}
        \nup(\Delta(E_{3, \fp})) = p \nup(\Delta(E_{2, \fp})) \text{ as } \nup(j(E_3)) = p \nup(j(E_2)).
    \end{align*}
    If $\nup(\Delta(E_{1, \fp})) \neq 0$, then $E_1$ and $E_3$ cannot be discriminant ideal twins. However, if $\nup(\Delta(E_{1, \fp})) = 0$, then we could potentially have discriminant twins, so let us consider $\nup(\Delta(E_{1, \fp}))$. Let $\nup(2) = f, \nup(3) = e, \nup(u_i) = \alpha_i$, and $\nup(d) = \delta \ge 0$. Then, using Table~\ref{ta:invariants}, we get the following:

$$
\arraycolsep=9pt
\def\arraystretch{1.1}
\begin{array}{rcc}
n=p^2 & \nup(c_{4, 1, \text{min}}) & \nup(\Delta(E_{1, \fp}))\\ \hline
9&4\alpha_1 + 4f + 2e + 2\delta- 4k&12\alpha_1 + 12 f + 6e + 6\delta- 11k\\\hline
25& 4f + 4e + 4\alpha_1 + 2\delta - 12k & 12f + 12e + 12\alpha_1 + 6\delta - 35k
\end{array}
$$

As the curve is integral at $\fp$, we must have $\nup(c_{4, 1, \text{min}}) \geq 0$. Consequently, 
\[
    \nup(\Delta(E_{i, \fp})) \ = \ 3\nup(c_{4, 1, \text{min}})+k \ \geq \ k \ > \ 0,
\]
so $\nup(\Delta(E_{1, \fp}))$ cannot be $0$. Therefore $E_1$ and $E_3$ cannot be discriminant ideal twins in the multiplicative reduction case. This is a contradiction to our assumption and hence $t \in \mathcal{O}_K$ in this case.  

Next, if the curves have potentially multiplicative reduction, then from  \cite[Table 1]{DokchitserDokchitser2015} we know that 
\begin{align*}
    \nup(\Delta(E_{2, \fp})) = \nup(\Delta(E_{1, \fp})) + (p-1)k \text{ as } \nup(j(E_2)) = p \nup(j(E_1))
\end{align*}
and
\begin{align*}
    \nup(\Delta(E_{3, \fp})) = \nup(\Delta(E_{2, \fp})) + p(p-1)k \text{ as } \nup(j(E_3)) = p \nup(j(E_2)).
\end{align*}
Therefore, 
\[
    \nup(\Delta(E_{3, \fp})) \ = \ \nup(\Delta(E_{1, \fp})) + (p^2-1)k \ \neq \ \nup(\Delta(E_{1, \fp})),
\] 
contradicting the assumption that $E_1$ and $E_3$ are discriminant ideal twins. It thus follows that $t \in \mathcal{O}_K$. 
\end{proof}

\subsection{Equal $j$-invariants}

In this section we categorize all $p^2$-isogenous discriminant (ideal) twins for $p = 3, 5$ where the two curves share the same $j$-invariant.
Note that the parameterized families of elliptic curves $\mathcal{C}_{p, i}(t,d)$ fail to capture the case when two $p^2$-isogenous elliptic curves share the same $j$-invariant $0$ or $1728$.
Thus, as with \cite{RNT1}, we break our theorems into two categories $j = 0, 1728$ and otherwise.
Recall from \cite[Lemma 3.1]{RNT1} that isogenous curves with the same $j$-invariant have complex multiplication (CM) and will be isomorphic over any field containing their CM endomorphism ring.
Thus they can only possibly be discriminant (ideal) twins over number fields that do not contain their CM endomorphism ring.

\remark{When the discriminant of a genus one curve is $0$, the curve is singular.  If $t_0$ is such that $\mathcal{C}_{n, i}(t_0, d)$ has discriminant $0$, then we call $t_0$ a singular value.}

\begin{theorem}
    Let $p \in \{3, 5\}$.
    There are no non-isomorphic $p^2$-isogenous curves that both have $j$-invariant $0$ or $1728$.
    \label{Thm:equaljinv_0_1728}
\end{theorem}

\begin{proof}
    Following \cite[Lemma 3.2]{RNT1}, we compute valuations of the modular polynomials $\Phi_{p^2}(j, y)$ for $j = 0, 1728$ and $p = 3, 5$ at $y-j$.
    Only $\Phi_{25}(1728, y)$ has any appropriate factors, $\nu_{y-1728}(\Phi_{25}(1728, y)) = 2$; for code see \cite{GitHubDIT}.
    Thus there are no non-isomorphic $9$-isogenous curves that both have $j$-invariant $0$ or $1728$ and there are no non-isomorphic $25$-isogenous curves that both have $j$-invariant $0$.
    This leaves us to check the $25$-isogeny case when both $j$-invariants are $1728.$

    Every cyclic $25$-isogeny $E_1 \rightarrow E_3$ factors as two $5$-isogenies: $E_1 \rightarrow E_2 \rightarrow E_3$.
    If all three curves have $j$-invariant $1728$, then by \cite[Proposition 3.4]{RNT1}, all three curves were actually isomorphic.
    To find the possible $j$-invariants for $E_2$ we use the $5$-isogeny parameterization found in \cite{Bariso} to set $j(\mathcal{C}_{5,1}(t, d)) = 1728$. We then solve for $t$ and determine the possible $j$-invariants for $\mathcal{C}_{5, 2}(t, d)$. The polynomial factors as $j(\mathcal{C}_{5,1}(t, d)) - 1728  = (t^2 + 22t + 125)(t^2 - 500t - 15625)^2/t^5$ with the roots of the factor $t^2 + 22t + 125$ corresponding to singular values.
    Thus we look at the roots of $t^2 - 500t - 15625,$ which are $t_0 = 250 \pm 125\sqrt{5}$. These give $j$-invariants $1728$ and $\alpha, \bar{\alpha} = 22015749613248 \pm 9845745509376\sqrt{5}$ for $\mathcal{C}_{5,1}(t_0, 1)$ and $\mathcal{C}_{5,2}(t_0, 1)$ respectively.
    Next, we check whether the curve $\mathcal{C}_{5,2}(t_0, 1)$ is $5$-isogenous to another elliptic curve with $j$-invariant $1728$ not isomorphic to $\mathcal{C}_{5,2}(t_0, 1)$.
    To do so, we reuse the trick from the beginning of this proof and factor the modular polynomial $\Phi_5(\alpha, y)$ to find $\nu_{y - 1728}(\Phi_5(\alpha, y)) = 1$.
    As any curve with $j$-invariant $\alpha$ only has one $5$-isogeny to a curve with $j$-invariant $1728$ over $\overline{K}$, there are no (non-isomorphic) $25$-isogenous curves that both have $j$-invariant 1728.
\end{proof}

\begin{theorem}
\label{Thm:equaljinv}
    Let $p = 3, 5$.
    If $E_1, E_3$ are $p^2$-isogenous discriminant ideal twins defined over a number field $K$ with $j = j(E_1) = j(E_3)$, but $j \neq 0$ or $1728$, then $j$ is as follows.
    If $p = 3$, then $t_0$ is a root of $t^2 - 6t - 18$ and $j(E_1) = j(E_3) = 14776832 t_0 + 32440512$ while $j(E_2) = 1728$ with CM endomorphism ring in $\mathbb{Q}(i)$.
    If $p = 5$, then $t_0$ is a root of $t^2 - 2t - 4$ and $j(E_1) = j(E_2) = 9845745509376 t_0 + 12170004103872$ and again $j(E_2) = 1728$ with CM endomorphism ring in $\mathbb{Q}(i)$.
\end{theorem}

\begin{proof}
The proof is exactly analogous to the $p$-isogenous case for $p = 2, 3, 5, 7, 13$ in \cite{RNT1}.
For $p^2 = 9, 25$, factor the parameterized $j$-invariants $j_1(t) - j_3(t)$.
For each factor $f(t)$, we let $\mathbb{Q}(t_0) \cong \mathbb{Q}[t]/(f(t))$ be the associated number field. If $t_0$ is not a singular value, then compute $E_1(t_0, 1), E_3(t_0, 1)$ and check if they are isomorphic over $\mathbb{Q}(t_0)$. 
If they are not isomorphic, then check if they are discriminant ideal twins and compute their CM field.

Following this algorithm in \cite{GitHubDIT}, we find that for $p^2 = 9$ the factors $t^2 + 3t + 9$ and $t - 3$ give singular values. 
The factor $t^2 - 6t - 18$ gives non-isomorphic curves over $\mathbb{Q}(t_0)$ that are discriminant twins.
The rest of the factors give isomorphic curves over $\mathbb{Q}(t_0)$. 
Similarly, for $p^2 = 25$ we find that the factors $t^4 + t^3 + 6t^2 + 6t + 11, t^2 + 4$ and $t - 1$ give singular values. 
The factor $t^2 - 2t - 4$ gives non-isomorphic curves that are discriminant twins.
As in the $p^2 = 9$ case, the rest of the factors give curves isomorphic over $\mathbb{Q}(t_0)$. 
Note that unlike the $p = 2, 3, 5, 7, 13$ cases in \cite{RNT1}, there are no non-isomorphic, non-discriminant ideal twin cases.
\end{proof}

\section{\texorpdfstring{Proof for $n=9$}{}}\label{proof_for_n=9}

In this section, we parameterize all $9$-isogenous discriminant ideal twins over number fields that do not share the same $j$-invariant.
As we classified all $9$-isogenous discriminant ideal twins over number fields that do share the same $j$-invariant, this completes the classification.
Our main tool is the $9$-isogeny elliptic curve parameterization due to Barrios \cite{Bariso}.  
Consider the polynomials $T_{m}$ for $1\leq m \leq 11$ in Table \ref{ta:polyTi}. Then for $i=1,3$ we can parameterize the two $9$-isogenous curves as in Table \ref{ta:curves}. Moreover,  the invariants $j_{i}$, $\Delta_i$, $c_{4,i}$, and $c_{6,i}$ are as given in Table \ref{ta:invariants}. 

\begin{theorem}\label{n=9_if}
  Let $E_1$ and $E_3$ be two $9$-isogenous elliptic curves over a number field $K$ with the property that their $j$-invariants are not equal. If $E_1$ and $E_3$ are discriminant ideal twins, then there exists $t\in \mathcal{O}_K$ and $d\in \mathcal{O}_K$ such that $E_i \cong \mathcal{C}_{9, i}(t, d)$, and for each prime
$\mathfrak{p}$ of $\mathcal{O}_K$, 
  \begin{equation}
  \label{eq:n9}
  \nup(t-3) = 3k_\fp,\qquad 0\leq k_\fp \leq \nup(3). 
\end{equation}
Moreover, if $E_1$ and $E_3$ are discriminant twins, then $t$ satisfies \eqref{eq:n9} and $(t-3)^8\in\mathcal{O}_{K}^{12}$.
\end{theorem}

\begin{proof} 
By Lemma \ref{Lem:tisOKint}, there exist $t \in \mathcal{O}_{K}$ and $d\in \mathcal{O}_K$ such that $E_{i}\cong \mathcal{C}_{p^2,i}(t,d)$, and thus we can say $E_{i}$ is given by a model 
\[y^2=x^3+d^2A_{9,i}(t)x+d^3B_{9,i}(t),\] with  $A_{9,i}(t), B_{9,i}(t)\in \mathcal O_K$ as defined in Table~\ref{ta:curves}. 
Let
$\mathfrak{D}_{i}^{\text{min}}$ denote the minimal discriminant ideal of
$E_{i}$; then $\mathfrak{D}_{1}^{\text{min}}=\mathfrak{D}_{3}^{\text{min}}$
since $E_{1}$ and $E_{3}$ are discriminant ideal twins.
Note that the models currently given for $E_1$ and $E_3$ are not necessarily minimal at a given prime $\fp$ and their discriminants $\Delta_1$ and $\Delta_3$ are not necessarily minimal at $\fp$ either.
As we are interested in the minimal discriminant ideal, we proceed prime by prime to construct $\mathcal{O}_K$-integral, $\fp$-minimal models of $E_1$ and $E_3$ to access and compare $\nup(\mathfrak{D}_{1}^{\text{min}})$ and $\nup(\mathfrak{D}_{3}^{\text{min}})$.

Let $\fp$ be a prime such that $\nup(3) \geq 1$. Then the signature of $E_3$ is of the form $(\geq 10\nup(3), \ \geq 15\nup(3), \ \geq 30\nup(3))$. In this case, by \cite[Tableau III]{Papadopoulos1993}, we know that the curve is not minimal at $\fp$. We can therefore make the change of variable $\tau_3 = [3^2, 0, 0, 0]$ to reduce this curve. From now on, replace $E_3$ by $\tau_3E_3$. 

For any prime $\mathfrak{p}$, we take $E_{i,\mathfrak{p}}$ to be a local minimal model of $E_{i}$ at $\mathfrak{p}$. Suppose further that $E_{i,\fp}$ is $\mathcal{O}_{K}$-integral.
Consequently, $\nu_{\mathfrak{p}%
}(\Delta(E_{1,\fp}))=\nu_{\mathfrak{p}}(\Delta(E_{3,\mathfrak{p}}))$, where $\Delta(E_{i, \fp})$ is the discriminant of the local minimal model $E_{i, \fp}$. 
In particular, there exists $\mu_{\mathfrak{p}}\in K$ such that
$\nu_{\mathfrak{p}}(\mu_{\mathfrak{p}})=0$ and $\Delta(E_{1,\mathfrak{p}})=\mu_{\mathfrak{p}}\Delta(E_{3,\mathfrak{p}})$. By
\cite[Proposition VII.1.3]{Silverman2009}, there exists $u_{i,\mathfrak{p}} \in\mathcal{O}_{K}$ such that $u_{i,\mathfrak{p}}^{-12}\Delta_i
=\Delta(E_{i,\mathfrak{p}})$. Consequently, $u_{1,\mathfrak{p}
}^{-12}\Delta_1=\mu_{\mathfrak{p}}u_{3,\mathfrak{p}}^{-12}\Delta_3$. Then 
\begin{equation}\label{mueqn9}
(t-3)^8=\frac{\Delta_1}{\Delta_3}=\mu_\fp\left(
\frac{u_{1,\fp}}{u_{3,\fp}}\right)^{12}. 
\end{equation}
Note that if $E_1$ and $E_3$ are discriminant twins, then $\mu_\fp = 1$. 

Let $\nup(T_{1}) = \nup(t-3) = s_\fp\geq 0$ and $\nup(3)=\ell\geq 0$. Suppose first that $E_1$ and $E_3$ are discriminant ideal twins.  
Applying $\nup(\cdot)$ to both sides of \eqref{mueqn9}, we have 
\[
8s_\fp = 12\nup(u_{1, \fp}/u_{3, \fp}).
\]
Note that since the valuation takes integral values, $s_\fp =3k_\fp$ for some $k_\fp \geq 0$. We want to show that $k_\fp\leq \ell$. In order to do this we investigate the possible values of $\nup(u_{1, \fp}/u_{3, \fp})$. 

We start by showing that we can locally take the twisting parameter $d$ up to squares.
For a prime $\fp$ of $\mathcal{O}_K$, let $\delta = \nup(d)$, and let $\pi \in \fp \setminus \fp^2$ be a uniformizer for $R_{\p}$. We may write $d = d_\p\pi^{2i_\p}$, where $\delta_\p = \nup(d_\p)\in\{0,1\}$. Then, using the change of variables $[\pi^{i_\p},0,0,0]$, we obtain the $\p$-integral model 
\[E_i': y^2 = x^3+d_\p^2A_{9,i}(t)x+d_\p^3B_{9,i}(t).\]
Note that if the signature of $E_i$ is $(\nup(c_{4,i}), \nup(c_{6,i}), \nup(\Delta_i))$, then the signature of $E_i'$ is $(\nup(c_{4,i}) - 4i_\p, \nup(c_{6,i}) - 6i_\p, \nup(\Delta_i) - 12i_\p)$ . As we are concerned only with the signature of a minimal model, we therefore consider the signature of $E_i'$ rather than $E_i$.
From hereon, we replace $E_i$ with $E'_i$.

First, we assume that $\ell = 0$. We will show that in this case $k_\fp=0$. Towards the contradiction, let us assume that $k_\fp\neq 0$. 
This forces $\nup(T_2) = \nup(T_3) = \nup(T_5)= \nup(T_7) = 0$.
Then the signatures of the two curves are: 
\begin{center}
    \begin{tabular}{ l l l }
        $\nup(c_{4, 1}) = 4\nup(2)+2{\delta_\p}$, & $\nup(c_{6, 1}) = 6\nup(2)+3{\delta_\p}$, & $\nup(\Delta_1) = 12\nup(2)+6{\delta_\p} + 9s_\fp$,  \\
        $\nup(c_{4, 3}) = 4\nup(2)+2{\delta_\p}$, & $\nup(c_{6, 3}) = 6\nup(2)+3{\delta_\p}$, & $\nup(\Delta_3) = 12\nup(2)+6{\delta_\p} + s_\fp$.
    \end{tabular}
\end{center}
If $\nup(2) = 0$, then the two curves are already minimal at $\fp$. 
If $\nup(2) \neq 0$, we can use \cite[Tableau V]{Papadopoulos1993} to see that $E_1$ and $E_3$ are minimal at $\fp$. In both cases, we observe that the two curves cannot be discriminant ideal twins (as $s_\fp=3k_\fp \neq 0$). This is a contradiction, hence $k_\fp=0$.

Next, we consider the case when $\ell \geq 1$. We want to show that $k_\fp\leq \ell$, i.e., $s_\fp\leq 3\ell$. Towards a contradiction, assume that $3\ell < s_\fp$. In this case, using Table \ref{Casesn=9}, we observe that the signatures of the two curves are:
     \begin{center}
        \begin{tabular}{ l l l }
            $\nu_\fp(c_{4, 1}) = 12\ell+2{\delta_\p}$, & $\nu_\fp(c_{6, 1}) = 18\ell+3{\delta_\p}$, & $\nu_\fp(\Delta_1) = 9\ell+6{\delta_\p}+9s_\fp$, \\ 
            $\nu_\fp(c_{4, 3}) = 4\ell+2{\delta_\p}$, & $\nu_\fp(c_{6, 3}) = 6\ell+3{\delta_\p}$, & $\nu_\fp(\Delta_3) = 9\ell+6{\delta_\p}+ s_\fp$. 
        \end{tabular}
     \end{center}
If we reduce the curve $E_3$ by $\tau_3=[\pi^i, 0, 0, 0]$  (for some integer $i \geq 0$) and the curve $E_1$ by $\tau_1 = [\pi^{i+2\ell}, 0, 0, 0]$, then the new reduced curves will have the following signatures:
\begin{center}
        \begin{tabular}{ l l l }
            $\nu_\fp(c_{4, 1}) = 4\ell+2{\delta_\p} - 4i$, & $\nu_\fp(c_{6, 1}) = 6\ell+3{\delta_\p} - 6i$, & $\nu_\fp(\Delta_1) = 9\ell+3k_\fp + 6{\delta_\p} - 12i + 24(k_\fp-\ell)$, \\ 
            $\nu_\fp(c_{4, 3}) = 4\ell+2{\delta_\p}- 4i$, & $\nu_\fp(c_{6, 3}) = 6\ell+3{\delta_\p}- 6i$, & $\nu_\fp(\Delta_3) = 9\ell+3k_\fp + 6{\delta_\p}-12i$. 
        \end{tabular}
     \end{center}
Observe that $\tau_1 E_1$ is an integral curve if and only if $\tau_3 E_3$ is an integral curve. Moreover, if we consider the value $i$ for which one of these curves is minimal, then using \cite[Tableau III]{Papadopoulos1993}, we observe that the other curve must also be minimal. It follows that these reductions give us minimal models $E_{1, \fp}$ and $E_{3, \fp}$. By noting that $\nup(\Delta(E_{1, \fp})) = 24(k_\p-\ell) + \nup(\Delta(E_{3,\fp})) > \nup (\Delta(E_{3, \fp}))$ (since $3k_\fp = s_\fp > 3\ell$), we conclude that in this case the two curves cannot be discriminant ideal twins, reaching a contradiction.

Hence we conclude that, for all primes $\fp$ of $\mathcal{O}_K$, $t$ satisfies
\begin{equation*}
  \nup(t-3) = 3k_\p, \quad \text{ for } 0\leq k_\p \leq \nup(3).
\end{equation*}

Lastly, suppose that $E_{1}$ and $E_{2}$ are discriminant twins. Then for each
prime $\mathfrak{p}$ of $\mathcal{O}_K$, we can take $\mu_{\mathfrak{p}}=1$ in
\eqref{mueqn9}. Consequently, we have that for $\mathfrak{p}$ with $\nup(3)\neq0$, we can
write%
\begin{equation}
(t-3)=\left(u_\fp\pi\right)
^{3},
\end{equation}
where $u_\fp$ is a unit in $\mathcal{O}_{K_{\mathfrak{p}}}$ and $\pi$ is a uniformizer. 
For prime $\mathfrak{p}$ with $\nup(3)=0$, it follows from \eqref{mueqn9} that there is a unit $u_\fp$  in $\mathcal{O}_{K_{\mathfrak{p}}}$ such that
\begin{equation}
(t-3)^8=u_\fp^{12}.
\end{equation}
Now consider the inclusion $K\hookrightarrow K_{\mathfrak{p}}$. Then $(t-3)^{8}\in\mathcal{O}_{K_{\mathfrak{p}}}^{12}$ for each prime $\fp$ of $K$.
 We conclude by the Grunwald-Wang
 Theorem \cite[Corollary 2]{MR33801} that $(t-3)^{8}\in\mathcal{O}_K^{12}$.
\end{proof}

\begin{theorem}\label{n=9_onlyif}
  Let $E_1$ and $E_3$ be two $9$-isogenous elliptic curves over a number field $K$ such that their $j$-invariants are not the same.
   If there exists $t \in \mathcal{O}_K$ and $d \in \mathcal{O}_K$  such that $E_{i}$ is
$K$-isomorphic to the elliptic curve $\mathcal{C}_{9,i}(t, d)$, and for
each prime $\mathfrak{p}$ of $\mathcal{O}_K$,
  \begin{equation}
  \label{DITn=9eq2}
    \nup(t-3) = 3k_\p, \qquad 0 \leq k_\p \leq \nup(3),
  \end{equation}
    then $E_1$ and $E_3$ are discriminant ideal twins. Moreover, if $t$ satisfies \eqref{DITn=9eq2} and  $(t-3)^{8}\in\mathcal{O}_{K}^{12}$,
    then $E_1$ and $E_3$ are discriminant twins.  
\end{theorem}

\begin{proof}
  Let $E_1$ and $E_3$ be two $9$-isogenous elliptic curves over a number field $K$. Suppose 
  that there exist $t \in \mathcal{O}_K$ and and $d\in \mathcal{O}_K$ such that $E_i \cong \mathcal{C}_{9, i}(t, d)$ with $t$ satisfying 
    \begin{align*}
        \nup(t-3) = 3k_\p=s_\fp,
  \end{align*}
  for some integer $0 \leq k_\p \leq \nup(3)$, for all primes $\fp$ of $\mathcal{O}_K$. Using the definitions of the invariants of the curve $E_3$ and \cite[Tableau III]{Papadopoulos1993}, we note that we can reduce $E_3$ at primes $\p$ above $3$. Thus, replace $E_3$ by $\tau_3 E_3$ where $\tau_3 = [3^2, 0, 0, 0]$. Let $E_{i,\fp}$ is a minimal model of $E_i$ at $\fp$.

Let $\p$ be a prime of $\mathcal{O}_K$. Let $\delta = \nup(d) \geq 0$. Note that we do not assume that $\delta \in \{ 0, 1 \}$. If $\nup(T_{1}) = \nup(t-3) > 0$, then by (\ref{DITn=9eq2}) we have $\nup(3) > 0$. In order to prove our statement, we consider all the cases given in Table~\ref{Casesn=9}. This table computes the valuations of the polynomials $T_m$ given in Table~\ref{ta:polyTi}, deriving the signatures and the valuation of the $j$-invariants of the two curves.
Note that the $j$-invariants satisfy $\nup(j_1) \geq 0$ and $\nup(j_3)\geq 0$ for all the cases in Table~\ref{Casesn=9} except for the case when $k_\fp=\ell/2$. Moreover, in all these cases, $\nup(\Delta_1)=\nup(\Delta_3)+24k_\fp$, 
i.e., $\nup(\Delta_1)\equiv\nup(\Delta_3) \mod {12}$. Then by \cite[Lemma 4.4]{RNT1}, we conclude that the $\nup(\Delta(E_{1, \p})) = \nup(\Delta(E_{3,\p}))$.

Next, we assume $k_\fp = \ell/2$. Let $\delta = \nup(d)$. 
Then, from Table~\ref{Casesn=9},
        \begin{center}
        \begin{tabular}{ l l l }
            $\nu_\fp(c_{4, 1}) = 8\ell+2\delta,$ & $\nu_\fp(c_{6, 1}) = 12\ell+3\delta,$ & $\nu_\fp(\Delta_1) = 9\ell+6\delta+\delta_3 +9s_\fp,$ \\
            
            $\nu_\fp(c_{4, 3}) = 4\ell+2\delta,$ & $\nu_\fp(c_{6, 3}) = 6\ell+3\delta,$ & $\nu_\fp(\Delta_3)= 9\ell+6\delta+\delta_3 +  s_\fp$.
        \end{tabular}
        \end{center}
    Here, $\delta_3$ is a non-negative integer. If we reduce the curve $E_3$ by $\tau_3 = [\fp^i, 0, 0, 0]$ and the curve $E_1$ by $\tau_1 =[\fp^{i+2 k_\fp}, 0, 0, 0]$ (for some integer $i\geq0$),
    then the new reduced curves will have equal signatures, given as follows: 
       $$(4\ell+2\delta-4i,\quad 6\ell+3\delta-6i,\quad 9\ell+6\delta+\delta_3 +  s_\fp-12i).$$ 
    Consider the value of $i$ for which $\tau_3 E_3$ is minimal. Then $\tau_1 E_1$ is also minimal (since they have the same signature). We conclude that $\nup(\Delta(E_{1, \fp})) = \nup(\Delta(E_{3,\fp}))$ in this case. 

Lastly, we assume that $\nup(T_{1}) = \nup(t-3) = 0$. We have the following cases: 
\begin{enumerate}
  \item[(a)] Assume $\nup(3) \neq 0$. In this case, the curves $E_1$ and $E_3$ have the same signature, so they reduce in the same way. It follows that $\nup(\Delta(E_{1, \p})) = \nup(\Delta(E_{3,\p}))$. 
  \item[(b)] Assume $\nup(3) = 0$ and $\nup(T_{3}) \neq 0$. Again, in this case, the curves $E_1$ and $E_3$ have the same signature, so they reduce in the same way. It follows that $\nup(\Delta(E_{1, \p})) = \nup(\Delta(E_{3,\p}))$. 
  \item[(c)] Assume $\nup(3) = 0$ and $\nup(T_{3}) = 0$.  
  If $\nup(2) = 0$, then $E_1$ and $E_3$ are already minimal at $\p$. If $\nup(2) \neq 0$, then using \cite[Tableau V]{Papadopoulos1993}, we can see that $E_1$ and $E_3$ are minimal at $\p$. Thus, $\nup(\Delta(E_{1, \p})) = \nup(\Delta(E_{3,\p}))$. 
\end{enumerate}

We conclude that $E_1$ and $E_3$ are discriminant ideal twins when $t$ satisfies \eqref{DITn=9eq2}.

It remains to show that if $t$ satisfies \eqref{DITn=9eq2} and $(t-3)^8\in\mathcal{O}_{K}^{12}$, then
$E_{1}$ and $E_{3}$ are discriminant twins. 
To this end, for a prime
$\mathfrak{p}$ of $K$, 
let $E_{i,\mathfrak{p}}$ be a $\mathfrak{p}$-minimal model of
$E_{i}$. 
By \cite[Proposition VII.1.3]{Silverman1994} there
exists $u_{i,\mathfrak{p}}\in\mathcal{O}_{K}$ such that $u_{i,\mathfrak{p}}^{-12}\Delta_i
=\Delta(E_{i,\mathfrak{p}})$. Moreover, since $E_{1}$
and $E_{3}$ are discriminant ideal twins (as they satisfy \eqref{DITn=9eq2}), it is the case that there is a
$\mu_{\mathfrak{p}}\in K$ with $\nup(\mu_{\mathfrak{p}})=0$ such that $\Delta(E_{1,\mathfrak{p}})=\mu_{\mathfrak{p}
}\Delta(E_{3,\mathfrak{p}})$. It follows that:
\[
(t-3)^8=\frac{\Delta_1}{\Delta_3}=\mu_\fp\left(
\frac{u_{1,\fp}}{u_{3,\fp}}\right)^{12}%
\qquad\Longrightarrow\qquad\mu_{\mathfrak{p}}=(t-3)^{8}\left(
\frac{u_{3,\fp}}{u_{1,\fp}}\right)^{12}.
\]
By our assumption, we get $\mu
_{\mathfrak{p}}\in\mathcal{O}_{K}^{12}$. In particular, there is a
$\kappa_{\mathfrak{p}}\in\mathcal{O}_{K}$ such that $\kappa_{\mathfrak{p}%
}^{12}=\mu_{\mathfrak{p}}$ and $\nu_{\mathfrak{p}}(\kappa_{\mathfrak{p}})=0$.
Now let $E_{1,\mathfrak{p}}^{\prime}$ be the elliptic curve obtained from
$E_{1,\mathfrak{p}}$ via the isomorphism $[\kappa_\fp,0,0,0]$. Then $E_{1,\mathfrak{p}}^\prime$ is an $\mathfrak{p}$-minimal
model  with $\Delta(E_{1,\mathfrak{p}}^{\prime}%
)=\kappa_{\mathfrak{p}}^{-12}\Delta(E_{1,\mathfrak{p}})$. Hence $\Delta
(E_{1,\mathfrak{p}})=\mu_{\mathfrak{p}}\Delta(E_{1,\mathfrak{p}}^{\prime})$,
which, in turn, yields that $\Delta(E_{1,\mathfrak{p}}^{\prime})=\Delta
(E_{3,\mathfrak{p}})$. This shows that for each $\mathfrak{p}$, there are
$\mathfrak{p}$-minimal models of $E_{1}$ and $E_{3}$ having equal
discriminants. Therefore $E_{1}$ and $E_{3}$ are discriminant twins.
\end{proof}

\section{\texorpdfstring{Proof for $n=25$}{}}\label{proof_for_n=25}

In this section, we parameterize all $25$-isogenous discriminant ideal twins over number fields that do not share the same $j$-invariant.
Again, as we have already classified all $25$-isogenous discriminant ideal twins over number fields that do share the same $j$-invariant, this completes the classification.
Our main tool is the $25$-isogeny elliptic curve parameterization due to Barrios \cite{Bariso}.  
Consider the polynomials $S_{m}$ for $1\leq m \leq 14$ in Table \ref{ta:polyTi}. Then for $i=1,3$ we can parameterize the two $25$-isogenous curves as in Table \ref{ta:curves}. 
Moreover, we can define the invariants $j_{i}$, $\Delta_i$, $c_{4,i}$, and $c_{6,i}$ as in Table \ref{ta:invariants}.

\begin{theorem}\label{n=25_if}
    Let $E_1$ and $E_3$ be two $25$-isogenous elliptic curves over $K$ with the property that their $j$-invariants are not the same. If $E_1$ and $E_3$ are discriminant ideal twins then there exists $t\in \mathcal{O}_K$ and $d\in \mathcal{O}_K$ such that $E_i \cong \mathcal{C}_{25, i}(t,d)$ and for each prime
$\mathfrak{p}$ of $K$, $t$ satifies
    \begin{equation}\label{eq1:DT=25}
        \nup(t-1) = k_\p,\qquad 0\leq k_\p \leq \nup(5).
    \end{equation}
\end{theorem}

\begin{proof}
By Lemma \ref{Lem:tisOKint}, there exist $t\in \mathcal{O}_K$ and $d\in \mathcal{O}_K$  such that $E_{i} \cong \mathcal{C}_{p^2,i}(t,d)$, and thus $E_{i}$ is given by an $\mathcal{O}_{K}$-integral model.
Now let
$\mathfrak{D}_{i}^{\text{min}}$ denote the minimal discriminant ideal of
$E_{i}$. Then $\mathfrak{D}_{1}^{\text{min}}=\mathfrak{D}_{3}^{\text{min}}$
since $E_{1}$ and $E_{3}$ are discriminant ideal twins.

Note that the curve $E_3$ is not minimal at the primes $\p$ above $5$, so we can make the change of variables $\tau_3 = [5^2, 0, 0, 0]$ to reduce this curve. From hereon, replace $E_3$ by $\tau_3 E_3$. 

For any prime $\mathfrak{p}$, let  $E_{i,\mathfrak{p}}$ denote a local minimal model of $E_{i}$ at $\mathfrak{p}$. Suppose further that $E_{i,\fp}$ is $\mathcal{O}_{K}$-integral.
Consequently, $\nu_{\mathfrak{p}%
}(\Delta(E_{1,\mathfrak{p}}))=\nu_{\mathfrak{p}}(\Delta(E_{3,\mathfrak{p}}))$.
In particular, there exists $\mu_{\mathfrak{p}}\in K$ such that
$\nu_{\mathfrak{p}}(\mu_{\mathfrak{p}})=0$ and $\Delta(E_{1,\mathfrak{p}}%
)=\mu_{\mathfrak{p}}\Delta(E_{3,\mathfrak{p}})$. By
\cite[Proposition VII.1.3]{Silverman1994}, there exists $u_{i,\mathfrak{p}}%
\in\mathcal{O}_{K}$ such that $u_{i,\mathfrak{p}}^{-12}\Delta_i
=\Delta(E_{i,\mathfrak{p}})$. Consequently, $u_{1,\mathfrak{p}
}^{-12}\Delta_1=\mu_{\mathfrak{p}}u_{2,\mathfrak{p}}^{-12}\Delta_3$. Then 
\begin{equation}\label{mueqn}
(t-1)^{24}=\frac{\Delta_1}{\Delta_3}=
\mu_{\mathfrak{p}}\left(
\frac{u_{1,\mathfrak{p}}}{u_{3,\mathfrak{p}}}\right)  ^{12}.
\end{equation}
Note that if $E_1$ and $E_3$ are discriminant twins, then $\mu_\p = 1$. 

Let $\nup(S_{3}) = \nup(t-1) = k_\fp\geq 0$ and $\nup(5)\geq 0$. Suppose first that $E_1$ and $E_3$ are discriminant ideal twins. Applying $\nup(\cdot)$ to both sides of \eqref{mueqn} implies that
\begin{equation*}
2k_p=2\nup(t-1) = \nup(u_{1, \p}/u_{3, \p}).
\end{equation*}
Now, we will show that $\nup(t-1)\leq \nup(5)$. For this we will investigate possible values of $\nup(u_{1, \p}/u_{3, \p})$ over all primes $\fp$ of $\mathcal{O}_K$.

As in the proof of Theorem~\ref{n=9_if}, we may write $d = d_{\p}\pi^{2i_{\p}}$ for an uniformizer $\pi \in \p \setminus \p^2$ for $R_\fp$ and $\delta_\p = \nup(d_\p) \in \{0, 1\}$. Just as in the proof of Theorem~\ref{n=9_if}, from hereon we may replace $E_i$ by the isomorphic $\p$-integral model $E_i^\prime$ using the change of variables $[\pi^{i_\p},0,0,0]$. From hereon, we can thus replace $\delta = \nup(d)$ by $\delta_\fp \in \{0, 1\}$.

If $\nup(\Delta_3) = 0$, then $\nup(\Delta_1) = 0$. So in this case $\nup(u_{1, \p}/u_{3, \p}) = 0$ and $k_\p=0$. We can therefore assume that $\nup(\Delta_3) > 0$.

We have two cases, depending on whether or not $\nup(5) = 0$. Firstly, assume that $\nup(5) = 0$. We want to show that $k_\fp=0$. Towards the contradiction, let us assume $k_\fp\neq0$. Then $\nu_\fp(S_{1}) = \nu_\fp(S_{2}) = \nu_\fp(S_{4}) = \nu_\fp(S_{5}) = \nu_\fp(S_{6}) = \nu_\fp(S_{7}) = \nu_\fp(S_{8}) = \nu_\fp(S_{9}) = 0$. Then the signatures of the two curves are: 
\begin{align*}
 &\text{sig}_{\fp}(E_1)=(4\nup(2)+4\nup(3)+2{\delta_\p}, 6\nup(2)+6\nup(3)+3{\delta_\p}, 12\nup(2)+12\nup(3)+6{\delta_\p} + 25k_\fp),\\   
&\text{sig}_{\fp}(E_3)=(4\nup(2)+4\nup(3)+2{\delta_\p}, 6\nup(2)+6\nup(3)+3{\delta_\p},12\nup(2)+12\nup(3)+6{\delta_\p} + k_\fp).
\end{align*}
If  $\nup(2)=\nup(3)=0$, then $\nu_\fp(c_{4, 1}) = \nu_\fp(c_{4,3}) = 2{\delta_p}$, so $E_1$ and $E_3$ are minimal at primes $\fp$ with $\nup(2)=\nup(3)=\nup(5)=0$.
If $\nup(2)>0$, then using \cite[Tableau V]{Papadopoulos1993}, we observe that the models of $E_1$ an $E_3$ are minimal at $\fp$ in this case as well.
If $\nup(3)>0$, then $\nu_\fp(\Delta_1) = \nu_\fp(\Delta_3)+24k_\fp$, $\nu_\fp(c_{4,1}) = \nu_\fp(c_{4,3})$, and $\nu_\fp(c_{6,1}) = \nu_\fp(c_{6,3})$.
Note that if we reduce the curve $E_3$ by $\tau_3 = [\fp^i, 0, 0, 0]$ and $E_1$ by $\tau_1 = [\fp^{i+2k_\fp}, 0, 0, 0]$ (for some integer $i \geq 0$), we obtain two new curves with the property that $\tau_1 E_1$ is an integral model if and only if $\tau_3 E_3$ is an integral model. Choose $i$ for which $\tau_3 E_3$ is a minimal model. By \cite[Tableau III]{Papadopoulos1993}, we observe that $\tau_1E_1$ is also a minimal model. In all these cases, we observe that the minimal models satisfy the fact that $\nup(\Delta(E_{1, \fp})) = \nup(\Delta(E_{3, \fp})) + 24k_\fp$, so the two curves cannot be discriminant ideal twins (since $k_\fp \neq 0$), which is a contradiction. 

Now, assume that $\ell=\nup(5) > 0$. In this case, $k_\fp\geq 0$. We want to show that 
$k_\fp\leq \ell$. Towards the contradiction, let us assume that $\ell < k_\fp$. Then by Table~\ref{Casesn=25},
\begin{center}
                \begin{tabular}{ l l l }
                    $\nu_\fp(c_{4, 1}) = 10\ell+2{\delta_\p},$ & $\nu_\fp(c_{6, 1}) = 15\ell+3{\delta_\p},$ & $\nu_\fp(\Delta_1) = 5\ell+6{\delta_\p}+25k_\fp,$ \\ 
                    $\nu_\fp(c_{4, 3}) = 2\ell+2{\delta_\p},$ & $\nu_\fp(c_{6, 3}) = 3\ell+3{\delta_\p},$ & $\nu_\fp(\Delta_3) = 5\ell+6{\delta_\p}+k_\fp$. 
                \end{tabular}
                \end{center}

Note that, in order to obtain minimal models, $E_3$ can be reduced by $\tau_3 = [\p^i, 0, 0, 0]$ if and only if $E_1$ can be reduced by $\tau_1 = [\p^{i + 2\ell}, 0, 0, 0]$ for some $i\geq 0$. In this case the minimal discriminants of $E_1$ and $E_3$ are not equal, since $k_\p\neq 0$. This is a contradiction. Hence we obtain that $t$ satisfies
\begin{equation*}
  \nup(t-1) = k_\p, \quad \text{ for } 0\leq k_\p \leq \nup(5).  
\end{equation*}
 \end{proof}

\begin{theorem}\label{n=25_onlyif}
    Let $E_1$ and $E_3$ be two $25$-isogenous elliptic curves over $K$ such that their $j$-invariants are not the same. If there exists $t \in \mathcal{O}_K$ and $d \in \mathcal{O}_K$  such that $E_{i}$ is
$K$-isomorphic to the elliptic curve $\mathcal{C}_{25, i}(t, d)$, and for
each prime $\mathfrak{p}$ of $\mathcal{O}_K$,
    \begin{equation}
    \label{eq3:25}
        \nup(t-1) = k_\p,\qquad 0 \leq k_\p \leq \nup(5),
    \end{equation}
    then $E_1$ and $E_3$ are discriminant ideal twins. Moreover, $E_1$ and $E_3$ are discriminant twins. 
\end{theorem}

\begin{proof}
    Let $E_1$ and $E_3$ be two $25$-isogenous elliptic curves over a number field $K$. Assume moreover that there exist $t \in \mathcal{O}_K$ and $d \in \mathcal{O}_K$ such that $E_i \cong \mathcal{C}_{25, i}(t, d)$ with $t$ satisfying the equation 
    \begin{equation*}
        \nup(t-1) = k_\p,
    \end{equation*}
    for some integer $0 \leq k_\p \leq \nu_\fp(5)$. The curve $E_3$ is not minimal at the primes dividing $5$, so we can replace $E_3$ by $\tau_3 E_3$, where $\tau_3 = [5^2, 0, 0, 0]$. Let $E_{i,\p}$ be a local minimal model of $E_i$ at $\fp$. 

Let $\p$ be a prime of $K$. If $k_\p=\nu_\fp(S_{3}) > 0$, then $\nup(5)>0$. In order to prove the statement, we consider the cases given in Table~\ref{Casesn=25}. This table computes the signatures of the curves $E_1$ and $E_3$, using the valuations of the polynomials $S_j$ in Table~\ref{ta:polyTi}. By
looking at the signatures in each case of Table~\ref{Casesn=25}, we observe that in order to get minimal models $E_{1,\fp}$ and $E_{3,\fp}$, the curve $E_3$ can be reduced by $\tau_3 = [\p^i, 0, 0, 0]$ if and only if the curve $E_1$ can be reduced by $\tau_1 = [\p^{i+2k_\fp}, 0, 0, 0]$. 
Then, using the fact that $\nu_\fp(\Delta_1) = 24k_\fp + \nu_\fp(\Delta_3)$,  we conclude that the $\nup(\Delta(E_{1,\p})) = \nup(\Delta(E_{3,\p}))$. 

If $\nu_\fp(S_{3})=k_\p= 0$, then we have the following cases: 
\begin{enumerate}
    \item[(a)] Assume $\nup(S_{3})=0$ and $\nup(S_{2})>0$. If $\nup(2)=\nup(3)=\nup(5)=0$, then the two curves are already reduced at $\p$, so $\nup(\Delta(E_{1,\p})) = \nup(\Delta(E_{3,\p}))$. It remains to consider the other cases. 
    \begin{enumerate}
        \item[(i)] Assume $\nup(2)>0$ or $\nup(3)>0$. In this case, $\nu_\fp(c_{4, 1}) = \nu_\fp(v_{4,3})$, $\nu_\fp(c_{6,1}) = \nu_\fp(c_{6,3})$ and $\nu_\fp(\Delta_1) = \nu_\fp(\Delta_3)$, so the two curves reduce in the same way. Thus, $\nup(\Delta(E_{1,\p})) = \nup(\Delta(E_{3,\p}))$. 
        \item[(ii)] Assume $\nup(5)>0$. Then $\nu_\fp(\Delta_1) = \nu_\fp(\Delta_3) = 3\nu_\fp(S_{2})$, $\nu_\fp(c_{4,1}) = \nu_\fp(c_{4,3}) = \nu_\fp(S_{2})$, and $\nu_\fp(c_{6,1}) \geq 2\nu_\fp(S_{2})$, $\nu_\fp(c_{6,3}) \geq 2\nu_\fp(S_{2})$, so the two curves reduce in the same way. Thus, $\nup(\Delta(E_{1,\p})) = \nup(\Delta(E_{3,\p}))$. 
    \end{enumerate}
    \item[(b)] Assume $\nup(S_{3})=\nup(S_{2})=0$ and $\nup(S_{1})>0$. If $\nup(5)=0$, then the two curves are minimal at $\p$. On the other hand, if $\nup(5)>0$, then $\nup(S_{3})>0$, reaching a contradiction. Thus, $\nup(\Delta(E_{1,\p})) = \nup(\Delta(E_{3,\p}))$. 
    \item[(c)] Assume $\nup(S_{1})=\nup(S_{2})=\nup(S_{3})=0$, but $\nup(2)>0$ or $\nup(3)>0$. In this case, the two curves are minimal at $\fp$ and $\nup(\Delta(E_{1,\p})) = \nup(\Delta(E_{3,\p}))$.
\end{enumerate}
Hence, we obtained that $E_1$ and $E_3$ are discriminant ideal twins when \eqref{eq3:25}
is satisfied. 

It remains to show that 
$E_{1}$ and $E_{3}$ are discriminant twins. 
To this end, for a prime
$\mathfrak{p}$ of $K$, 
let $E_{i,\mathfrak{p}}$ be a $\mathfrak{p}$-minimal model of
$E_{i}$. 
By \cite[Proposition VII.1.3]{Silverman1994}, there
exists $u_{i,\mathfrak{p}}\in\mathcal{O}_{K}$ such that $u_{i,\mathfrak{p}}^{-12}\Delta_i
=\Delta(E_{i,\mathfrak{p}})$. Moreover, since $E_{1}$
and $E_{3}$ are discriminant ideal twins it is the case that there is a $\mu_{\mathfrak{p}}\in K$
such that $\nu_{\mathfrak{p}}
(\mu_{\mathfrak{p}})=0$ and $\Delta(E_{1,\mathfrak{p}})=\mu_{\mathfrak{p}
}\Delta(E_{3,\mathfrak{p}})$. It follows that
\[
(t-1)^{24}=\frac{\Delta_1}{\Delta_3}=\mu_\fp\left(
\frac{u_{1,\fp}}{u_{3,\fp}}\right)^{12}%
\qquad\Longrightarrow\qquad\mu_{\mathfrak{p}}=(t-1)^{24}\left(
\frac{u_{3,\fp}}{u_{1,\fp}}\right)^{12}.
\]
Note that $\mu
_{\mathfrak{p}}\in\mathcal{O}_{K}^{12}$ since $(t-1)^{24}\in\mathcal{O}_{K}^{12}$. In particular, there is a
$\kappa_{\mathfrak{p}}\in\mathcal{O}_{K}$ such that $\kappa_{\mathfrak{p}%
}^{12}=\mu_{\mathfrak{p}}$ and $\nu_{\mathfrak{p}}(\kappa_{\mathfrak{p}})=0$.
Now let $E_{1,\mathfrak{p}}^{\prime}$ be the elliptic curve obtained from
$E_{1,\mathfrak{p}}$ via the isomorphism $[\kappa_\fp,0,0,0]$. Then $E_{1,\mathfrak{p}}^\prime$ is a $\mathfrak{p}$-minimal
model  with $\Delta(E_{1,\mathfrak{p}}^{\prime}%
)=\kappa_{\mathfrak{p}}^{-12}\Delta(E_{1,\mathfrak{p}})$. Hence $\Delta
(E_{1,\mathfrak{p}})=\mu_{\mathfrak{p}}\Delta(E_{1,\mathfrak{p}}^{\prime})$,
which, in turn, yields that $\Delta(E_{1,\mathfrak{p}}^{\prime})=\Delta
(E_{3,\mathfrak{p}})$. This shows that for each $\mathfrak{p}$, there are
$\mathfrak{p}$-minimal models of $E_{1}$ and $E_{3}$ having equal
discriminants. Therefore $E_{1}$ and $E_{3}$ are discriminant twins.
\end{proof}

\section{Explicit results over \texorpdfstring{$\mathbb{Q}$}{Q}
}\label{sec:Examples}

\begin{proposition}
\label{classinQ}Up to twists, there are finitely many $p^2$-isogenous
discriminant ideal twins over~$\mathbb{Q}$, for odd $p$ for which $X_0(p^2)$ has genus $0$. These elliptic curves, $E_{1}$ and $E_{3}$, are given in Table~\ref{ta:clasoverQ}
by their LMFDB label. The table also lists their $j$-invariant and minimal discriminant.
\end{proposition}

{\renewcommand*{\arraystretch}{1.2}
\renewcommand{\arraycolsep}{0.4cm}
\begin{longtable}[c]{ccccccc}
\caption{$p^2$-isogenous discriminant ideal twins over $\mathbb{Q}$}\\
\hline
$p^2$ & $E_{1}$ & $E_{3}$ & $j\!\left(  E_{1}\right)  $ & $j\!\left(
E_{3}\right)  $ & $\Delta_{E_{1}}^{\text{min}}$ & $\Delta_{E_{3}}^{\text{min}}$ \\
\hline
\endfirsthead
\caption[]{\emph{continued}}\\
\hline
$p^2$ & $E_{1}$ & $E_{3}$ & $j\!\left(  E_{1}\right)  $ & $j\!\left(
E_{3}\right)  $ & $\Delta_{E_{1}}^{\text{min}}$ & $\Delta_{E_{3}}^{\text{min}}$\\
\hline
\endhead
\hline
\multicolumn{6}{r}{\emph{continued on next page}}
\endfoot
\hline
\endlastfoot

$9$ & $\href{http://www.lmfdb.org/EllipticCurve/Q/37.b1/}{37.b1}$ & $\href{http://www.lmfdb.org/EllipticCurve/Q/37.b2/}{37.b2}$ & $\frac{727057727488000}{37}$ & $\frac{4096000}{37}$ & $37$ & $37$\\\cmidrule(lr){2-7}
& $\href{http://www.lmfdb.org/EllipticCurve/Q/171.b1/}{171.b1}$ & $\href{http://www.lmfdb.org/EllipticCurve/Q/171.b3/}{171.b3}$ & $\frac{-50357871050752}{19}$ & $\frac{32768}{19}$ & $-13851$ & $-13851$\\\hline
$25$ & $\href{http://www.lmfdb.org/EllipticCurve/Q/11.a1/}{11.a1}$ & $\href{http://www.lmfdb.org/EllipticCurve/Q/11.a3/}{11.a3}$ & $\frac{-52893159101157376}{11}$ & $\frac{-4096}{11}$ & $-11$ & $-11$\\\cmidrule(lr){2-7}
& $\href{http://www.lmfdb.org/EllipticCurve/Q/18176.e1/}{18176.e1}$ & $\href{http://www.lmfdb.org/EllipticCurve/Q/18176.e3/}{18176.e3}$ & $\frac{3922540634246430781376}{71}$ & $\frac{190705121216}{71}$ & $36352$ & $36352$
\label{ta:clasoverQ}	
\end{longtable}
}

\begin{proof}
    Suppose $E_1$ and $E_3$ are two $p^2$-isogenous elliptic curves that are discriminant ideal twins over $\mathbb{Q}$, for $p = 3, 5$. By Theorem~\ref{Thm:equaljinv_0_1728}, we can assume that the two curves have $j$-invariants not both identically $0$ or $1728$. By Lemma \ref{Lem:tisOKint}, we know that there exists $t \in \mathbb{Z}$ and $d \in \mathbb{Z}$ such that $E_i \cong \mathcal{C}_{p^2, i}(t, d)$.

    If $p = 3$, then Theorem~\ref{n=9_iff} says that $\nu_q(t-3) = 3k_q$, for $0 \leq k_q \leq \nu_q(3)$, for all rational primes $q$. This means that $t-3$ is either a unit in $\mathbb{Q}$ or a unit multiplied by $3^3$. Since the only units in $\mathbb{Q}$ are $\pm 1$, it follows that $t - 3 = \pm1$ or $t-3 = \pm 3^3$, which implies that $t = 2, 4, -24, 30$. Note that $j(\mathcal{C}_{9, 1}(\pm 1 + 3, d)) = j(\mathcal{C}_{9, 3}(\pm 27 + 3, d))$, so $\mathcal{C}_{9, 3}(\pm 27 + 3, d)$ is a twist of $\mathcal{C}_{9, 1}(\pm 1 + 3, d)$. Therefore, we only have to consider the pairs for which $t = \pm 1 + 3$. As $d$ is a twisting parameter, for $t=4$, we obtain the curves $\href{http://www.lmfdb.org/EllipticCurve/Q/37.b1/}{37.b1}$ and $\href{http://www.lmfdb.org/EllipticCurve/Q/37.b3/}{37.b3}$; for $t = 2$, we obtain the curves $\href{http://www.lmfdb.org/EllipticCurve/Q/171.b1/}{171.b1}$ and $\href{http://www.lmfdb.org/EllipticCurve/Q/171.b3/}{171.b3}$. 

    If $p = 5$, then Theorem~\ref{n=25_iff} says that $\nu_q(t-1) = k_q$, for $0 \leq k_q \leq \nu_q(5)$, for all rational primes $q$. This means that $t-3$ is either a unit in $\mathbb{Q}$ or a unit multiplied by $5$. Since the only units in $\mathbb{Q}$ are $\pm 1$, it follows that $t - 1 = \pm1$ or $t-1 = \pm 5$, which implies that $t = 0, 2, -4, 6$. Note that $j(\mathcal{C}_{25, 1}(\pm 1 + 1, d)) = j(\mathcal{C}_{25, 3}(\pm 5 + 1, d))$, so $\mathcal{C}_{25, 3}(\pm 5 + 1, d)$ is a twist of $\mathcal{C}_{25, 1}(\pm 1 + 1, d)$. Therefore, we only have to consider the pairs for which $t = \pm 1 + 1$. As $d$ is a twisting parameter, for $t=0$, we obtain the curves $\href{http://www.lmfdb.org/EllipticCurve/Q/11.a1/}{11.a1}$ and $\href{http://www.lmfdb.org/EllipticCurve/Q/11.a3/}{11.a3}$; for $t = 2$, we obtain the curves $\href{http://www.lmfdb.org/EllipticCurve/Q/18176.e1/}{18176.e1}$ and $\href{http://www.lmfdb.org/EllipticCurve/Q/18176.e3/}{18176.e3}$. 
\end{proof}

\newpage
\appendix
\section*{Appendix: Tables of valuations}
Let $E_1$ and $E_3$ be two $p^2$-isogenous elliptic curves, for $p = 3, 5$. It follows from the work of Barrios \cite{Bariso} that there exist elements $t, d \in K$ such that $E_i \cong \Ci$, where the curves $\Ci$ are defined in Table \ref{ta:curves}. For each $p = 3, 5$, let $\fp$ be a prime of $K$. Let $\delta = \nup(d)$. Moreover, let 
\[ 
\ell = \left\{
\begin{array}{cc}
    \nup(3) & \text{ if } p = 3 \\
    \nup(5) & \text{ if } p = 5 \\
\end{array}
\right. 
\]
and
\begin{align*}
s &= \nup(t-3) \quad \text{ if } p = 3 \\
k &= \nup(t-1) \quad \text{ if } p = 5
\end{align*}
Note that $k = k_\fp$ in Theorems~\ref{n=25_if} and \ref{n=25_onlyif} and $s = s_\fp$ in Theorems~\ref{n=9_if} and \ref{n=9_onlyif}.
Assume that $k > 0$, $s > 0$, and $\ell > 0$ and define the following tables.
Note that in these tables, the values $\delta_i$ (for $1\leq i \leq 8$) are integers with the property that $\delta_i \geq 0$, for all $i$. 

\begin{center}
\begin{longtable}{c c c c c c}
\caption{Cases for $n=9$.} \label{Casesn=9}\\
        \hline 
        Poly. & $s < \ell/3$ & $s=\ell/3$ & $\ell/3 < s < \ell/2$ & $s = \ell/2$ & $\ell/2 < s < \ell$  \\
		\hline
		\endfirsthead
		\caption{Cases for $n=9$} \\
		\endhead
		\hline
		\multicolumn{4}{r}{\emph{continued on next page}}
		\endfoot
		\hline
		\endlastfoot	
        \hline 
        $T_{1}$ & $s$ & $s$ & $s$ & $s$ & $s$ \\ \cmidrule(lr){2-6}
        $T_{2}$ & $s$ & $s$ & $s$ & $s$ & $s$ \\ \cmidrule(lr){2-6}
        $T_{3}$ & $2s$ & $2s$ & $2s$ & $2s$ & $2s$ \\ \cmidrule(lr){2-6}
        $T_{4}$ & $3s$ & $3s + \delta_4$ & $\ell$ & $\ell$ & $\ell$ \\ \cmidrule(lr){2-6}
        $T_{5}$ & $3s$ & $3s$ & $3s$ & $3s$ & $3s$ \\ \cmidrule(lr){2-6}
        $T_{6}$ & $6s$ & $6s$ & $6s$ & $3\ell + \delta_6$ & $3\ell$ \\ \cmidrule(lr){2-6}
        $T_{7}$ & $6s$ & $6s$ & $6s$ & $6s$ & $6s$ \\ \cmidrule(lr){2-6}
        $T_{8}$ & $s$ & $s$ & $s$ & $s$ & $s$ \\
        \cmidrule[\heavyrulewidth](lr){1-6}
        $c_{4,1}$ & $2\ell + 4s + 2\delta$ & $10s + 2\delta$ & $2\ell + 4s + 2\delta$ & $8s + 2\delta$ & $2\ell + 4s + 2\delta$ \\ \cmidrule(lr){2-6}
        $c_{6,1}$ & $3\ell + 6s + 3\delta$ & $15s + 3\delta$ & $3\ell + 6s + 3\delta$ & $12s + 3\delta$ & $3\ell + 6s + 3\delta$ \\ \cmidrule(lr){2-6}
        $\Delta_1$ & $6\ell + 11s + 6\delta$ & $29s + 6\delta$ & $6\ell + 11s + 6\delta$ & $23s + 6\delta$ & $6\ell + 11s + 6\delta$ \\ \cmidrule(lr){2-6}
        $j_1$ & $s$ & $s$ & $s$ & $s$ & $s$ \\
        \cmidrule[\heavyrulewidth](lr){1-6}
        $c_{4,3}$ & $2\ell + 4s + 2\delta$ & $10s + \delta_4 + 2\delta$ & $3\ell +s + 2\delta$ & $7s + 2\delta$ & $3\ell + s + 2\delta$ \\ \cmidrule(lr){2-6}
        $c_{6,3}$ & $3\ell + 6s + 3\delta$ & $15s + 3\delta$ & $3\ell +6s + 3\delta$ & $12s + \delta_6 + 3\delta$ & $6\ell + 3\delta$ \\ \cmidrule(lr){2-6}
        $\Delta_3$ & $6\ell + 3s + 6\delta$ & $21s + 6\delta$ & $6\ell + 3s + 6\delta$ & $15s + 6\delta$ & $6\ell + 3s + 6\delta$ \\ \cmidrule(lr){2-6}
        $j_3$ & $9s$ & $9s + 3\delta_4$ & $3\ell$ & $6s$ & $3\ell$ \\
        \hline
        \hline
        \newpage
        Poly. & $\ell = s$ & $\ell < s < 3\ell/2$ & $s=3\ell/2$ & $3\ell/2 < s < 2\ell$ & $s = 2\ell$  \\
        \hline 

        \hline
        $T_{1}$ & $s$ & $s$ & $s$ & $s$ & $s$  \\ \cmidrule(lr){2-6}
        $T_{2}$ & $s$ & $s$ & $s$ & $s$ & $2\ell + \delta_2$  \\ \cmidrule(lr){2-6}
        $T_{3}$ & $2s$ & $2s$ & $3\ell + \delta_3$ & $3\ell$ & $3\ell$  \\ \cmidrule(lr){2-6}
        $T_{4}$ & $\ell$ & $\ell$ & $\ell$ & $\ell$ & $\ell$  \\ \cmidrule(lr){2-6}
        $T_{5}$ & $3s$ & $3s$ & $3s$ & $3s$ & $3s$  \\ \cmidrule(lr){2-6}
        $T_{6}$ & $3\ell$ & $3\ell$ & $3\ell$ & $3\ell$ & $3\ell$  \\ \cmidrule(lr){2-6}
        $T_{7}$ & $6s$ & $6s$ & $6s$ & $6s$ & $6s$  \\ \cmidrule(lr){2-6}
        $T_{8}$ & $\ell + \delta_8$ & $\ell$ & $\ell$ & $\ell$ & $\ell$  \\
        \cmidrule[\heavyrulewidth](lr){1-6}
        $c_{4,1}$ & $6s + 2\delta$ & $2\ell + 4k + 2\delta$ & $8\ell + 2\delta$ & $2\ell + 4s + 2\delta$ & $10\ell+\delta_2 + 2\delta$  \\ \cmidrule(lr){2-6}
        $c_{6,1}$ & $9s + 3\delta$ & $3\ell + 6s + 3\delta$ & $12\ell+3\delta$ & $3\ell + 6s+3\delta$ & $15\ell+3\delta$  \\ \cmidrule(lr){2-6}
        $\Delta_1$ & $17s +6\delta$ & $6\ell + 11s+6\delta$ & $9\ell + 9s + \delta_3+6\delta$ & $9\ell + 9s+6\delta$ & $27\ell+6\delta$  \\ \cmidrule(lr){2-6}
        $j_1$ & $s$ & $s$ & $s-\delta_3$ & $3s-3\ell$ & $3\ell + 3\delta_2$  \\
        \cmidrule[\heavyrulewidth](lr){1-6}
        $c_{4,3}$ & $4s + \delta_8 + 2\delta$ & $4\ell+2\delta$ & $4\ell+2\delta$ & $4\ell+2\delta$ & $4\ell+2\delta$  \\ \cmidrule(lr){2-6}
        $c_{6,3}$ & $6s+3\delta$ & $6\ell+3\delta$ & $6\ell+3\delta$ & $6\ell+3\delta$ & $6\ell+3\delta$  \\ \cmidrule(lr){2-6}
        $\Delta_3$ & $9s+6\delta$ & $6\ell + 3s+6\delta$ & $9\ell + s + \delta_3+6\delta$ & $9\ell + s+6\delta$ & $11\ell+6\delta$  \\ \cmidrule(lr){2-6}
        $j_3$ & $3s + 3\delta_8$ & $6\ell - 3s$ & $s-\delta_3$ & $3\ell - s$ & $\ell$  \\
        \hline
        \hline

        Poly. & $2\ell < s < 5\ell/2$ & $s = 5\ell/2$ & $5\ell/2 < s < 8\ell/3$ & $s = 8\ell/3$ & $8\ell/3 < s$ \\
        \hline 
        \hline 
        $T_{1}$ & $s$ & $s$ & $s$ & $s$ & $s$  \\ \cmidrule(lr){2-6}
        $T_{2}$ & $2\ell$ & $2\ell$ & $2\ell$ & $2\ell$ & $2\ell$  \\ \cmidrule(lr){2-6}
        $T_{3}$ & $3\ell$ & $3\ell$ & $3\ell$ & $3\ell$ & $3\ell$  \\ \cmidrule(lr){2-6}
        $T_{4}$ & $\ell$ & $\ell$ & $\ell$ & $\ell$ & $\ell$ \\ \cmidrule(lr){2-6}
        $T_{5}$ & $3s$ & $3s$ & $3s$ & $8\ell + \delta_5$ & $8\ell$  \\ \cmidrule(lr){2-6}
        $T_{6}$ & $3\ell$ & $3\ell$ & $3\ell$ & $3\ell$ & $3\ell$  \\ \cmidrule(lr){2-6}
        $T_{7}$ & $6s$ & $15\ell + \delta_7$ & $15\ell$ & $15\ell$ & $15\ell$  \\ \cmidrule(lr){2-6}
        $T_{8}$ & $\ell$ & $\ell$ & $\ell$ & $\ell$ & $\ell$  \\
        \cmidrule[\heavyrulewidth](lr){1-6}
        $c_{4,1}$ & $4\ell + 3s+2\delta$ & $4\ell + 3s+2\delta$ & $4\ell + 3s+2\delta$ & $12\ell + \delta_5+2\delta$ & $12\ell+2\delta$  \\ \cmidrule(lr){2-6}
        $c_{6,1}$ & $3\ell + 6s+3\delta$ & $18\ell + \delta_7+3\delta$ & $18\ell+3\delta$ & $18\ell+3\delta$ & $18\ell+3\delta$  \\ \cmidrule(lr){2-6}
        $\Delta_1$ & $9\ell + 9s+6\delta$ & $9\ell + 9s+6\delta$ & $9\ell + 9s+6\delta$ & $9\ell + 9s+6\delta$ & $9\ell + 9s+6\delta$  \\ \cmidrule(lr){2-6}
        $j_1$ & $3\ell$ & $3\ell$ & $3\ell$ & $3\ell + 3\delta_5$ & $27\ell - 9s$  \\
        \cmidrule[\heavyrulewidth](lr){1-6}
        $c_{4,3}$ & $4\ell+2\delta$ & $4\ell+2\delta$  & $4\ell+2\delta$ & $4\ell+2\delta$ & $4\ell+2\delta$  \\ \cmidrule(lr){2-6}
        $c_{6,3}$ & $6\ell+3\delta$ & $6\ell+3\delta$ & $6\ell+3\delta$ & $6\ell+3\delta$ & $6\ell+3\delta$  \\ \cmidrule(lr){2-6}
        $\Delta_3$ & $9\ell + s+6\delta$ & $9\ell + s+6\delta$ & $9\ell + s+6\delta$ & $9\ell + s+6\delta$ & $9\ell + s+6\delta$  \\ \cmidrule(lr){2-6}
        $j_3$ & $3\ell - s$ & $3\ell - s$ & $3\ell - s$ & $3\ell - s$ & $3\ell - s$ \\
        \hline
\end{longtable}
\end{center}

\begin{center}
\begin{longtable}{ C{0.2in} C{0.7in} C{0.8in} C{0.7in} C{0.7in}  C{0.8in}  C{0.6in}  C{0.6in}}
\caption{Cases for $n=25$} \label{Casesn=25}\\
    \hline 
    Poly. & $\ell < k$ & $\ell = k$ & $k < \ell < 2k$ & $\ell = 2k$ & $2k < \ell < 10k$ & $\ell = 10k$  & $10k < \ell$ \\ 
    \hline
		\endfirsthead
		\caption{Cases for $n=25$} \\
		\endhead
		\hline
		\multicolumn{7}{r}{\emph{continued on next page}}
\endfoot
		\hline
		\endlastfoot	
    \hline 
    $S_{1}$ & $2\ell$ & $2\ell$ & $2\ell$ & $2\ell + \delta_1$ & $4k$ & $4k$ & $4k$  \\ 
    \cmidrule(lr){2-8}
    $S_{2}$ & $\ell$ & $\ell$ + $\delta_2$ & $k$ & $k$ & $k$ & $k$ & $k$  \\ 
    \cmidrule(lr){2-8}
    $S_{3}$ & $k$ & $k$ & $k$ & $k$ & $k$ & $k$ & $k$  \\ 
    \cmidrule(lr){2-8}
    $S_{4}$ & $9\ell$ & $9\ell$ & $10k$ & $10k$ & $10k$ & $10k$ & $10k$  \\ 
    \cmidrule(lr){2-8}
    $S_{5}$ & $3\ell$ & $3\ell + \delta_5$ & $2k+\ell$ & $4k + \delta_5$ & $4k$ & $4k$ & $4k$  \\ 
    \cmidrule(lr){2-8}
    $S_{6}$ & $10\ell$ & $10\ell + \delta_6$ & $10k$ & $10k$ & $10k$ & $10k$ & $10k$  \\ 
    \cmidrule(lr){2-8}
    $S_{7}$ & $\ell$ & $\ell$ & $\ell$ & $\ell$ & $\ell$ & $\ell + \delta_7$ & $10 k$ \\ 
    \cmidrule(lr){2-8}
    $S_{8}$ & $\ell$ & $\ell$ & $\ell$ & $\ell + \delta_8$ & $2k$ & $2k$ & $2k$  \\ 
    \cmidrule(lr){2-8}
    $S_{9}$ & $0$ & $0$ & $0$ & $0$ & $0$ & $0$ & $0$ \\ 
    \hline
    
    \hline
    $c_{4,1}$ & $10\ell+2\delta$ & $10k + \delta_2+2\delta$ & $11k+2\delta$ & $11k+2\delta$ & $11k+2\delta$ & $11k+2\delta$ & $11k+2\delta$  \\ 
    \cmidrule(lr){2-8}
    $c_{6,1}$ & $15 \ell+3\delta$ & $15k + 2\delta_2 + \delta_5 + \delta_6+3\delta$ & $14k + \ell+3\delta$ & $16k + \delta_5+3\delta$ & $16k+3\delta$ & $16k+3\delta$ & $16k+3\delta$ \\ 
    \cmidrule(lr){2-8}
    $\Delta_1$ & $5\ell + 25k+6\delta$ & $30k + 3\delta_2+6\delta$ & $2\ell + 28k+6\delta$ & $32k + \delta_1+6\delta$ & $32k+6\delta$ & $32k+6\delta$ & $32k+6\delta$ \\ 
    \cmidrule(lr){2-8}
    $j_1$ & $25\ell - 25k$ & $0$ & $5k-2\ell$ & $k - \delta_1$ & $k$ & $k$ & $k$  \\
    \hline
    
    \hline
    $c_{4,3}$ & $2\ell+2\delta$ & $2k+\delta_2+2\delta$ & $k + \ell+2\delta$ & $3k+2\delta$ & $k + \ell+2\delta$ & $11k + \delta_7+2\delta$ & $11k+2\delta$   \\ 
    \cmidrule(lr){2-8}
    $c_{6,3}$ & $3\ell+3\delta$ & $3k + 2\delta_2+3\delta$ & $2k + \ell+3\delta$ & $4k + \delta_8+3\delta$ & $4k+3\delta$ & $4k+3\delta$ & $4k+3\delta$ \\ 
    \cmidrule(lr){2-8}
    $\Delta_3$ & $5\ell + k+6\delta$ & $6k + 3\delta_2+6\delta$ & $2\ell + 4k+6\delta$ & $8k + \delta_1+6\delta$ & $8k+6\delta$ & $8k+6\delta$ & $8k+6\delta$  \\ 
    \cmidrule(lr){2-8}
    $j_3$ & $\ell - k$ & $0$ & $\ell - k$ & $k - \delta_1$ & $3\ell - 5k$ & $25k + 3\delta_7$ & $25k$  \\
    \hline
\end{longtable}
\end{center}

\bibliographystyle{amsplain}
\bibliography{Win6}
\end{document}